\documentclass[a4paper,english, 11pt]{amsart}
\makeatletter
 \theoremstyle{plain}
\newtheorem{thm}{Theorem}[section]
\theoremstyle{plain}
  \newtheorem{prop}[thm]{Proposition}
\theoremstyle{plain}

\theoremstyle{plain}
 \newtheorem{lemma}[thm]{Lemma}
\theoremstyle{plain}

\theoremstyle{plain}

\theoremstyle{definition}
  
 \theoremstyle{definition}

\theoremstyle{remark}

\numberwithin{equation}{section}
 
\usepackage{amsmath}
\usepackage{amssymb}
\usepackage{amscd}
\usepackage{amsthm}
\usepackage{array}
\usepackage{xcolor}
\usepackage[arrow,matrix,tips,all,cmtip]{xy}

\usepackage{stmaryrd}
\usepackage{hyperref}
\usepackage{comment}

\pdfpageheight\paperheight
\pdfpagewidth\paperwidth
\newcommand{\Z}{\mathbb{Z}}

\newcommand{\Q}{\mathbb{Q}}

\newcommand{\C}{\mathbb{C}}

\newcommand{\F}{\mathbb{F}}

\newcommand{\fa}{\mathfrak{a}}
\newcommand{\fb}{\mathfrak{b}}
\newcommand{\fc}{\mathfrak{c}}

\newcommand{\fd}{\mathfrak{d}}

\newcommand{\fM}{\mathfrak{M}}
\newcommand{\fN}{\mathfrak{N}}

\newcommand{\fS}{\mathfrak{S}}

\newcommand{\fw}{\mathfrak{w}}

\newcommand{\fA}{\mathfrak{A}}
\newcommand{\fm}{\mathfrak{m}}

\newcommand{\cA}{\mathcal{A}}

\newcommand{\cM}{\mathcal{M}}

\newcommand{\cO}{\mathcal{O}}

\newcommand{\cS}{\mathcal{S}}

\newcommand{\eps}{\varepsilon}

\newcommand{\ku}{k[\![u]\!]}

\newcommand{\Gal}{\mathrm{Gal}}
\newcommand{\Hom}{\mathrm{Hom}}

\newcommand{\Ind}{\mathrm{Ind}}

\newcommand{\GL}{\mathrm{GL}}

\newcommand{\Spec}{\mathrm{Spec}}

\newcommand{\cris}{\rm{cris}}
\newcommand{\upi}{{\underline \pi}}
\newcommand{\Vbar}{{\overline V}}

\DeclareMathOperator{\coker}{coker}

\DeclareMathOperator{\Fil}{Fil}

\DeclareMathOperator{\st}{{st}}

\DeclareMathOperator{\rM}{{M}}

\newcommand{\Qpbar}{\overline{\Q}_p}

\newcommand{\wt}{\widetilde}

\newcommand{\M}{{\mathfrak M}}

  \usepackage{mathrsfs}
  \usepackage{multicol}

\makeatother

\usepackage{enumitem}

\title[Reductions of crystalline representations]{Reductions of certain 3-dimensional crystalline representations}

\author{Tong Liu}
\address{Department of Mathematics,
Purdue University, 
	 150 N. University Street, 
	West Lafayette, Indiana 47907, USA}
\email{tongliu@math.purdue.edu}
\date{\today}


\begin{document}

\begin{abstract}
We compute reduction  $\bar \rho $ of 3-dimensional irreducible crystalline representations $\rho$ of $G_{\Q_p}$ with Hodge-Tate weights $\{0, r , s\}$ satisfying $2 \leq r \leq p-2, \ \  2+p \leq s \leq r + p-2.$ If $\bar \rho$ is irreducible then we show that the generic fiber of crystalline universal deformation ring of $\bar \rho$ is connected. 
\end{abstract}
\maketitle

\setcounter{tocdepth}{1}
\tableofcontents

\section{Introduction}

The aim of this paper is to explicitly compute reduction of certain family of 3-dimensional crystalline representations. Let $V$ be an \emph{irreducible} 3-dimensional crystalline representation of $G_{\Q_p}$ with coefficient field $F$. Suppose that Hodge -Tate weights of  $V$ are $\{0, r , s\}$, which satisfies the following \emph{generic} condition: 
$$2 \leq r \leq p-2, \ \  2+p \leq s \leq r + p-2. $$
Let $T \subset V$ be a $ G$-stable $\cO_F$-lattices, our aim is to decide the \emph{semi-simiplification} $\overline  V $ of reduction $T/ \varpi T$, where $\varpi$ is a fixed uniformizer of $\cO_F$. Our method follows the idea in \cite{BergdallLevin-BLZ}, \cite{BLL-semistable}  which calculate the integral Kisin module $\M(T)$ for $T$ and then obtain $\overline V$ via $\M (T) / \varpi \M(T)$. More precisely,   we first give explicit classifications of $V$ in term of strongly divisible lattice inside weakly admissible filtered $\varphi$-module $D= D^*_{\rm cris} (V)$: There are 3 Types with some explicit parameters (see Theorem \ref{them-classification-SD-lattices}). Then we are able to obtain Kisin module $\cM (T)$ over $S= \cO_F[\![u, \frac{E^p}{p}]\!]$ with $E= u +p$ by the strategy described in \cite{BLL-semistable}. The difficult step is to descend $\cM(T)$ to a Kisin module $\fM$ over $\fS : = \cO_F[\![u]\!]$. This involve quite complicated calculation to select suitable basis of $\cM(T)$ so that the matrix of Frobenius $\varphi$ for this basis have entries inside $\fS$.  To this end, we obtain many situations of reduction type of $\overline V$ which depends on the valuations of certain function of parameters for these 3 types. See \S \ref{subsec-reducible-red} and  \S \ref{subsec-irreducible-red} for more details. 

 Finally, we use these reduction types to study universal crystalline deformation ring constructed in \cite{kisin4}. Let $\bar \rho : G_{\Q_p} \to \GL_3 (\F)$ be a residue representation. If $\bar \rho$ is irreducible then by \cite{kisin4} for Hodge-Tate type ${\bf v}: =\{0 , r, s\}$ there exists the universal crystalline deformation ring $R_{\bar \rho} ^{\bf v}$.

\begin{thm}[Theorem \ref{thm-deformation}]\label{thm-deformation-mian} Let  $\rho: G_{\Q_p} \to \GL_3 (\cO_F)$ be a crystalline representation with Hodge-Tate weights $\{0, r, s\}$ satisfying $2 \leq r \leq p-2, \ \  2+p \leq s \leq r + p-2. $	Assume that the reduction $\bar \rho: = \rho \mod \varpi $ is irreducible. Then $\Spec( R^{\bf v}_{\bar \rho}[\frac 1 p ])$ is connected. 
\end{thm}

\subsection*{Acknowledgment:} The author would like to thank for Toby Gee, Daneil Lei and Brandon Levin  for useful discussion and comments when prepare this paper. 

\subsection{Notations} We use $[a_1 , \dots , a_n]$ to denote $n \times n$-diagonal matrix with $a_1, \dots , a_n$ on the main diagonal. Let $R$ be a ring then ${\rm M}_d (R)$ (resp. $\GL_d (R)$) denote the ring (resp. group) of (resp. invertible) $d\times d$-matrices with entries in $R$.  Suppose that $\varphi : R \to R$ be a ring endomorphism. Pick $X\in \GL_n (R)$ and $A\in {\rm M}_d (R)$. We write $X*_\varphi A = X A \varphi (X^{-1})$ and $X* A = X A X^{-1}$.  We use $v_p (\cdot)$ to denote $p$-adic valuation so that $v_p (p) =1$. We reserve  $E = u +p \in \fS$ and write $\gamma_i (E): = (\frac{E^p} p )^i. $
 
\section{Classification of Irreducible Strongly Divisible Lattices} 

Let $F$ be a finite extension of $\Q_p$. Write $\cO= \cO_F$ the ring of integers. Fix an uniformizer $\varpi \in F$. 
Let $V$ be an irreducible 3-dimensional crystalline representation of $G_{\Q_p}$ with coefficient field $F$ and Hodge-Tate weights $\{0, r, s\}$. 
Recall that $V$ is classified by weakly admissible filtered $\varphi$-module $D= D^*_{\rm cris} (V) = \Hom_{\Q_p[G_{\Q_p}]} (V , B^+_{\cris})$. Furthermore, by \cite{Laffaille} there exists a \emph{strongly divisible lattice}  $M \subset D$, in that sense that $M$ is a finite free $\cO_F$-submodule of $D$ satisfying
\begin{enumerate}
	\item $F \otimes_{\cO_F} M = D$; 
	\item set $\Fil ^i M : = M \cap\Fil ^ i D $. Then $\varphi(\Fil ^ i M)\subset p ^i M; $
	\item $\sum_{i}\frac{\varphi}{p ^i}(\Fil ^ i M)= M$. 
\end{enumerate}
Our first aim is to provide explicit matrices to describe such $M$ and simplify such matrix if possible. 
In this section, we only assume that $0 < r <s$. 
By the above discussion, we may select $\cO$-basis $e_1 , e_2 , e_3$ so that $e_2 \in \Fil ^r M $ and $e_3 \in \Fil ^s M$. Therefore, we obtain a relation 
$$\varphi (e_1, e_2 , e_3) = (e_1, e_2 , e_3)A [1, p ^r, p ^s]$$ where $A = (a_{ij})_{3 \times 3} \in \GL_3 (\cO)$ and $[1, p ^r , p ^s]$ denotes the diagonal matrix with $1, p ^r$ and $p ^s $ in the main diagonal. If we want to change basis of $M$, say $(f_1, f_2, f_3) = (e_1, e_2 , e_3) X^{-1 }$ with $X \in \GL_3 (\cO)$, then $X$ can be only \emph{lower triangular} because we need $f_2 \in \Fil ^r M$ and $f_3 \in \Fil ^s M$.  Note that 
$$\varphi (f_1 , f_2 , f_3)= (f_1 , f_2 , f_3) X A[1, p^r , p ^s] X^{-1}. $$
Write $\Lambda = [1, p ^ r, p ^s]$. Since  $X$ is lower triangular, it is easy to check that $ X* A\Lambda = A ' \Lambda$ with $A' \in \GL_3 (\cO)$. 
Our goal in this section  is to find $X$ so that $ A'$ is as simple as possible. 
\subsection{Elementary operations} Now let us discuss some special types $X$ which will be extensively used in the following. 
Let $S_i (c)$ denote matrix which induces row operation multiplying $c$ to $i$-th row. It is clear that $ S_i (c)* A\Lambda$ will multiply $c$ in $i$-row but $c^{-1}$ to $i$-th columns of $A\Lambda$. When $c \in F$  (resp.  $c\in \cO^\times$) then $S_i (c)^* A\Lambda$ corresponds to change basis in $M [\frac 1 p]$  (resp.  $M$).

Let $R_{ij}(a)$ denote the matrix which  induces a row operation adding $a$-multiple of $i$-th row to the $j$-th row.  Note that $ R_{ij} (a)^{-1}$ induces a column operation which adding $-a$-multiple of $j$-th column to $i$-th column. 
For each $j$, define $\tilde i_j  : = \max\{i| v_p (a_{ij})= 0 \}$.  Now for each fixed $j$, if $i > \tilde i _j$ then $v_p (a_{ij})>0$. 
 In the following, we always use $R_{ \tilde i _j i } (-\frac{a_{ij}}{a_{\tilde i_j j}})$ to kills $ij$-th entry of $XA\Lambda$. 
 
\begin{lemma}\label{lem-eliminate} For each fixed $j$,  there exists a change of basis $M$ so that  $ a_{ij}= 0$ for any $i > \tilde i _j$. 
\end{lemma} 
\begin{proof} Pick $m _j=  \min \{ v_p (a_{ij})|  i > \tilde i_{j}\}$ and $l_j= \# \{k | k > \tilde i _j, v_p (a_{kj}) = m_j\}$. We construct a sequence a basis so that $l_j$ drops to $0$ and then  $m_j$ can go to $+ \infty$.  Pick $a_{k j}$ so that $v_p (a_{kj})= m _j> 0 $. 
	Consider the matrix $X= R_{\tilde i _j k} (-  \frac {a_{kj}}{a_{\tilde i_j j}})$.   
	By definition, $k > \tilde i_j$. So $X$ is low-triangular as required. Also $X A \Lambda$ makes $kj$-th entry $0$. On the other hand,  $X* (A \Lambda) = XA\Lambda X^{-1}$ will add  $\frac {a_{kj}}{a_{\tilde i_j j}}$-multiple of $k$-th column back $\tilde i _j$-th column. Unless $\tilde i_j =j$, the $j$-th column of $XA\Lambda X^{-1}$ is the same as that of  $X A \Lambda$. 
		 If $\tilde i_j =j$ then we see that $j$-th column of $XA\Lambda X^{-1}$  is the $j$-th column of adding  $\frac {a_{kj}}{a_{\tilde i_j j}}$-multiple of $k$-th column of $XA\Lambda$. Write $X*(A \Lambda)= A' \Lambda $ with $A'= (a_{ij}')_{3 \times 3}$. 
		 Since   $v_p (\frac {a_{kj}}{a_{\tilde i_j j}}) =m_j$ and $k$-th column of $XA\Lambda$ has factor $p ^{k}$, and note $k > j$, we conclude that $a'_{\tilde i _j j}$ is still a unit, $v_p (a'_{ij})> m_j$ if $v_p (a_{ij})> m_j$ and $v_p (a'_{kj}) \geq m _j + k-\tilde i _j > m _j$.  Therefore, if we replace $A\Lambda$ by $X*(A\Lambda)$ then we see that $l_j$ decrease at least 1 to $l_j -1$. This completes the proof. 
\end{proof}

Now by the above Lemma to our case, we immediately see $v_p (a_{11})> 0$. Otherwise, after change a basis,  $a_{21}=a _{31}= 0$. Then $A \Lambda$ represents a reducible representation. Contradiction. 

Now assume that $a_{21}\in \cO ^\times$. We claim that $v_p(a_{12}) > 0$. Otherwise, we claim that after change a basis $a_{31} = a_{32}= 0$ so that $V$ is reducible and contradiction. Now assume that $a_{12} \in \cO ^\times$, we can use the same trick as the above Lemma: Let $ h = \min \{ v_p (a_{31}), v_p (a_{32})\}$. 
 Taking $X= R_{23} (- \frac{a_{31}}{a_{21}}) R_{13} (-\frac{a_{32}}{a_{12}})$, consider $X* A \Lambda : = A' \Lambda$ with $A' = (a'_{ij})_{3\times 3}$. It is not hard to check that 
 $a'_{21}$ and $a'_{12}$ remain unit and $\min\{v_p (a'_{31}), v_p (a'_{32})\}\geq h + 1$. 
 So we keep this process and obtain a basis so that    $a_{31} = a_{32}= 0$, and contradicts that $V$ is irreducible. This concludes that $v_p (a_{12})> 0$. So $a_{13}$ must be a unit because $A\in \GL _3 (\cO). $

  In summary, if $a_{21} \in \cO ^\times$ then so is $a_{13}$ and both $a_{11}$ and $a_{12}$ are not units. Using $X= R_{12} (-\frac {a_{23}}{a_{13}}) R_{13} (- \frac{a_{33}}{a_{13}})$, by replacing $A\Lambda$ by$ X*A \Lambda$,  $A$ has the shape 
  $A= \begin{pmatrix}
a_{11} & a_{12} & a_{13} \\ a_{21} & a_{22}  & 0 \\ a_{31} & a_{32} & 0 
  \end{pmatrix}$, 
  where  $a_{21}$ and $a_{13}$  are still units,  and both $a_{11}$ and $a_{12}$ are not. Now using $X= R_{23} (-\frac{a_{31}}{a_{21}})$, and replacing $A\Lambda$ by $X* A\Lambda$, we have $A= \begin{pmatrix}
  	a_{11} & a_{12} & a_{13} \\ a_{21} & a_{22}  & 0 \\ 0 & a_{32} & 0 
  \end{pmatrix}$, where  $a_{21}$ and $a_{13}$  are still units,  and both $a_{11}$ and $a_{12}$ are not. This forces $a_{32}\in \cO^\times$. Now there are two cases: 
  
  Case I: $v_p(a_{22})> 0$. We have   $a_{21}$, $a_{23}$ $a_{13}$  are still units, but   $a_{11}$,  $a_{12}$ and $a_{22}$ are not. 
  
  Case II $a_{22} \in \cO ^\times$. In this case, we use $X= R_{23} (- \frac{a_{32}}{a_{22}})$ to kill $a_{32}$ entry, but since $a_{21}, a_{22}, a_{32}$ are all units, we see that 31-entry of $X* A\Lambda$ returns a unit. By replacing $A\Lambda$ by $X* A \Lambda$, we have 
   $A= \begin{pmatrix}
   a_{11} & a_{12} & a_{13} \\ a_{21} & a_{22}  & 0 \\ a_{31} & 0 & 0 
   \end{pmatrix}$
  where $a_{21}, a_{31}, a_{22}$ and $a_{13}$ are units, but $a_{11}$ and $a_{12}$ are not. 
  The above discussion finishes the case that $a_{21}\in \cO^\times$. 
  
  Now let us assume that $a_{31}\in \cO ^\times$ but $a_{11}, a_{21}$ are not. Since $A\mod \varpi$ is still invertible, this forces that either $a_{13}$ or $a_{23}$ are units.

   Case II':  $a_{13} \in \cO ^\times$. In this case, we can use the above situation when $a_{13}\in \cO ^\times$.  Using $X= R_{12} (-\frac {a_{23}}{a_{13}}) R_{13} (- \frac{a_{33}}{a_{13}})$, by replacing $A\Lambda$ by$ X*A \Lambda$, then 
   $A= \begin{pmatrix}
   a_{11} & a_{12} & a_{13} \\ a_{21} & a_{22}  & 0 \\ a_{31} & a_{32} & 0 
   \end{pmatrix}$, 
   where  $a_{31}$ and $a_{13}$  are still units,  and both $a_{11}$ and $a_{21}$ are not. This forces that $a_{22}\in \cO^\times$.  Using $X= R_{23} (- \frac{a_{32}}{a_{22}})$ to kill $a_{32}$ entry, and now $A$ becomes 
    $A= \begin{pmatrix}
  a_{11} & a_{12} & a_{13} \\ a_{21} & a_{22}  & 0 \\ a_{31} & 0 & 0 
  \end{pmatrix}$ where $a_{31}, a_{22}, a_{13}$ are units, $a_{11}$, $a_{21}$ are not. 
  
Case III:  $a_{13}$ is not unit.  Then this forces  $a_{23}$ and $a_{12}$ are units. Now using $X = R_{23} (-\frac {a_{33}}{a_{23}})$ to kills $33$-entry, we have new   $A= \begin{pmatrix}
a_{11} & a_{12} & a_{13} \\ a_{21} & a_{22}  & a_{23} \\ a_{31} & a_{32} & 0 
\end{pmatrix}$ where $a_{31}$,  $a_{12}$ and $a_{23}$ are still units, and $a_{11}, a_{21}$ are not. Now using  $X= R_{13}(-\frac{a_{32}}{a_{12}}) R_{12} (- \frac{a_{22}}{a_{12}})$ to kills $22$, $32$-entries, we obtains
  $$A= \begin{pmatrix}
a_{11} & a_{12} & a_{13} \\ a_{21} & 0   & a_{23} \\ a_{31} & 0 & 0 
\end{pmatrix}$$
where $a_{31}, a_{12}$ and $a_{23}$ are units, but $a_{11}, a_{21}$ and $a_{13}$ are not.

Now we claim that Case II and Case II' can be realized in the same type in the following sense: For case II, $v_p (a_{12})> 0$, by enlarging $F$ if necessary (enlarging $F$ will not affect computing reduction), we may assume that $\varpi ^2 | a_{12}$. Now consider $S_2 (\varpi)* A \Lambda$, we can see then both $a_{12}$ and $a_{21}$ not units. Similarly trick can be applied to case II' so that to make both $a_{12}$ and $a_{21}$ to be non-units. So we may put case II and II' into one type  whose has matrix 
 $A= \begin{pmatrix}
a_{11} & a_{12} & a_{13} \\ a_{21} & a_{22}  & 0 \\ a_{31} & 0 & 0 
\end{pmatrix}$ where $a_{31}, a_{22}, a_{13}$ are units, $a_{11}$, $a_{12}$,  $a_{21}$ are not. Note in this type $a_{12} a_{21} \not = 0$. Indeed, if $a_{21}= 0$ (resp. $a_{12}= 0$) then $e_1 , e_3$  (resp. $e_2$) generates weakly admissible filtered submodule of $D$. This contradicts that $V$ is irreducible.

By enlarging $F$ so that $F$ contains the cubic roots of $\det(A)$,  We can replace  $A $ by  $\alpha A$ so that  $\alpha \in F$ and $\det (A) =1$. For case I, consider $ S_1 ((a_{13} a_{32})^{-1}) S_3 (a_{32}^{-1}) *A\Lambda$, $A$ is simplified to following shape: 
  $A= \alpha \begin{pmatrix}
a & b & 1 \\ 1 & c  & 0 \\ 0 & 1 & 0 
\end{pmatrix}, \alpha \in \cO^\times, a, \ b ,\ c  \in \varpi \cO. $ 
The same trick applies to case III, and we get 
  $A= \alpha \begin{pmatrix}
a & 1 & b  \\ c & 0   & 1 \\ 1 & 0  & 0 
\end{pmatrix}, \alpha \in \cO^\times, a, \ b ,\ c  \in \varpi \cO $. For case II, we consider $S_1 (a_{13}^{-1})\Lambda$, then $A$ is simplified to 
 $A= \alpha \begin{pmatrix}
a & b  & 1 \\ c & \mu   & 0 \\ -\mu^{-1} & 0  & 0 
\end{pmatrix}, \  \alpha,  \mu \in \cO^\times,  \ a, b,  c\in \varpi \cO.$

Finally we call matrices in case I type I,  and case III Type II, but case II type III because type I and type II are more essential in compute the reduction of $V$. We summarize our discussion in this subsection to the following theorem. 
\begin{thm}\label{them-classification-SD-lattices} To compute (the semi-simplification of) reduction of irreducible 3-dimensional crystalline representations of $G_{\Q_p}$ with Hodge-Tate weights $\{ 0 , r , s\}$ satisfying $0 < r < s$, it suffices to consider the representations whose $D^*_{\cris} (V)$ containing strongly divisible lattices given by the following 3 types of matrices $A \Lambda$,  where $\Lambda = [1, p ^ r , p ^s]$ and $A$ has the following types: 
	$$ \begin{pmatrix}
	a & b & 1 \\ 1 & c  & 0 \\ 0 & 1 & 0 
	\end{pmatrix},  a,  b , c  \in \varpi \cO;     \begin{pmatrix}
	a & 1 & c \\ b & 0   & 1 \\ 1 & 0  & 0 
	\end{pmatrix},  a,  b , c  \in \varpi \cO;    	
	\begin{pmatrix}
	a & b  & 1 \\ c & \mu   & 0 \\ -\mu^{-1} & 0  & 0 
	\end{pmatrix},   \mu \in \cO^\times,  \ a, b,  c\in \varpi \cO, bc\not = 0. 
	$$
\end{thm}

\subsection{Relations between matrices of different types.}\label{subsec-relations}

Indeed, the type I and type II are dual to each other up to twisting characters. Let $V$ be given by type II and $V ^* : = \Hom_{F} (V , F)$ be the dual of $V$. Then $V^*(s)$ is also crystalline with Hodge-Tate weights $\{ 0 , s-r, s\}$.  It is standard to compute that the corresponding  strongly divisible lattice in $ V^* (s)$ is given by $ (A^{T}) ^{-1} p ^s \Lambda ^{-1}$. After switching basis so that the matrix has formation $A' [1, p ^{s-r}, p ^s]$, we have $A': = \begin{pmatrix}
	-b   & bc -a  & 1  \\ 1 & -c   & 0  \\ 0 & 1  & 0 
	\end{pmatrix}. $
And this is actually type I representation. 	Finally, we remark the generic condition $2 \leq r \leq p-2, 2+p \leq s \leq r + p-2$ is stable under the dual. Namely, if $r, s $ satisfies this condition, then so does the Hodge-Tate weights $\{0 , s-r, s\}$ of $V^*(s)$. 

\subsubsection{${\rm{III}} (a, b, c)= {\rm{I}} (a, bc - p ^{s-r} \mu ^{-1}, \mu )$}\label{rmk-III-I} Type III can be transferred to type I: Since $c\not = 0$, we have 
	 $\begin{pmatrix} 1 & 0 & 0 \\ 0 & c ^{-1} & 0 \\ 0 & 0 & 1\end{pmatrix}*_\varphi A= 	\begin{pmatrix}
	 a & bc  & 1 \\ 1 & \mu   & 0 \\ -\mu^{-1} & 0  & 0 
	 \end{pmatrix}: = A' $. Then $ \begin{pmatrix} 1 & 0 & 0 \\ 0 & 1 & 0 \\ 0 & \mu ^{-1} & 1\end{pmatrix}*_\varphi A'\Lambda= 	\begin{pmatrix}
	 a & bc - p^{s-r}\mu^{-1}  & 1 \\ 1 & \mu   & 0 \\ 0 & 1  & 0 
	 \end{pmatrix}\Lambda $. Note that both $v_p (a), v_p (bc - p ^{s-r}\mu ^{-1}) >0 $. So Type III can be regarded as Type I by allowing $c \in \cO^\times$.   	
\subsubsection{${\rm I}(a, b , c) = {\rm II} (a_2 , b_2 , c_2) $}\label{subsubsec-I to II}
Suppose we are give type I matrix 
$A\Lambda= \begin{pmatrix} a & b p ^r & p ^s \\ 1 & c p ^r & 0 \\ 0 & p ^r & 0 \end{pmatrix}.$ Let us assume that $c\not = 0 $ and $b + c^{-1}p ^{s-r}\not = 0$. 
Then $\begin{pmatrix}
	1 & 0 & 0 \\ 0 & 1 & 0 \\ 0 & 0 & c
\end{pmatrix} *_\varphi (A\Lambda) =  \begin{pmatrix} a & b p ^r & c ^{-1}p ^s \\ 1 & c p ^r & 0 \\ 0 & c p ^r & 0 \end{pmatrix}$. By 
$\begin{pmatrix}
	1 & 0 & 0 \\ 0 & 1 & 0 \\ 0 & -1 & 1
\end{pmatrix} *_\varphi$, we obtain $ \begin{pmatrix} a & (b + c^{-1} p ^{s-r})  p ^r & c ^{-1}p ^s \\ 1 & c p ^r & 0 \\ -1 & 0 & 0 \end{pmatrix}$. Using 
$\begin{pmatrix}
	1 & 0 & 0 \\ 0 & (b+ c^{-1}p ^{s-r})  & 0 \\ 0 & 0 & 1
\end{pmatrix} *_\varphi$, we obtain $ \begin{pmatrix} a &  p ^r & c ^{-1}p ^s \\ (b + c^{-1} p ^{s-r})  & c p ^r & 0 \\ -1 & 0 & 0 \end{pmatrix}.$ Finally by  
$\begin{pmatrix}
	1 & 0 & 0 \\ - c  & 1   & 0 \\ 0 & 0 & -1
\end{pmatrix} *_\varphi$,  we obtain $$ \begin{pmatrix} a + cp ^r &  p ^r & c ^{-1}p ^s \\ (b -ac+  c^{-1} p ^{s-r})  & 0  & p ^ s\\ 1 & 0 & 0 \end{pmatrix}.$$ 
Hence ${\rm I}(a, b , c )= {\rm II }(a_2, b_2, c_2)$ with $a_2 = a + c p ^r$ , $b_2 = (b -ac+  c^{-1} p ^{s-r}),$ and $c_2 = c^{-1}$. Note here we change basis by rational coefficients. 

Combine the above result, since $bc\not = 0$,  ${\rm III} (a, b , c)= {\rm II} (a+ \mu p ^r, b c - a \mu, \mu^{-1})$. 
\section{Algorithm and Validations}
In this section, we establish an algorithm to compute $\overline \M: = \M \mod \varpi$ for a Kisin module $\M$ over $\fS:=  \cO_F [\![u]\!]$ 
attached to crystalline representation $V$ in this paper. For the background of Kisin module $\M$ attached to crystalline representations, please see \cite[\S 2]{BLL-semistable}. The following subsection will make a summary for that  paper.

\subsection{Summary from \cite{BLL-semistable}}\label{subsec-summary}Let $S = \cO [\![u, \frac{E^p}{p}]\!]$ and $\nu : = \frac {\varphi(E)}{p}$\footnote{Note the notation $\nu$ is different from that in \cite{BLL-semistable} where use $\fc$.}. Note that $\nu\in S ^\times$. Let $V$ be a semistable representation $V$ with Hodge-Tate weights in $[0, \dots, h]$ and $G_K$-stable $\cO_F$-lattice $T\subset V$. Let $\M : = \M (T)$ be the Kisin module attached to $T$. Then \cite[Thm. 2.7]{BLL-semistable} provides a method to compute the matrix of $\varphi$  on $\cM : = S [\frac 1 p] \otimes _{\fS} \M$  via $D^*_{\st}  (V)$. In the case that  $K = \Q_p$ and $V$ is crystalline,  the method is  explicitly explained in  Example 2.9 in \cite{BLL-semistable}:   By \cite{Laffaille}, $D= D_{\st}^{\ast}(V)= D_{\cris}^*(V)$ admits a strongly divisible lattice  $(M , \Fil ^i M, \varphi _i)$. More precisely, there exists an $F$-basis $(e_1, \dots , e_d)$ of $D$ and integers $0 =n_0 \leq n_1 \leq \cdots \leq n_h \leq d$ such that $\Fil ^ i D : = \bigoplus_{j\geq n_i } F e_j $, and $\varphi (e_1,\dots , e_d) = (e_1, \dots , e_d) X P$ where $X \in \GL_d (\cO)$ and $P$ is a diagonal matrix whose $ii$-th entry is $p^{s_i}$ where $s_i= \max\{j \mid n_j \leq i\}= \max \{ j \mid e_i \in \Fil ^j D\}$. Then there exists an $S[\frac 1 p]$-basis of $\cM $ so that the matrix $\cA $ of $\varphi$ on $\cM$ is given by \begin{equation}\label{eqn-compute-phi}
\cA=  \Lambda X C, \text{ where }  \  \Lambda = [E^{s_1}, \dots , E^{s_d}], \  C = [\nu ^{-s_1}, \dots ,\nu ^{ - s_d} ]. 
\end{equation} 

 In the next subsection, we will use Equation \eqref{eqn-compute-phi} to compute $\cA$ for strongly divisible lattices in Theorem \ref{them-classification-SD-lattices}.  Let $R : = \cO[\![ u,  \frac{\varpi E}{p}]\!]$. Let $\cS$ be any ring satisfying $\fS \subset \cS \subset F [\![u]\!]$ and $h \geq 0$, we use $$\rM^h_d (\cS): = \{ A \in \rM_d (\cS)| \exists B \in \rM _d (\cS), AB = BA = E^h I_d\}. $$ In the remaining of this paper, we construct a sequence matrices $X_n\in \GL_3 (R[\frac 1 p])$ so that $Y : = \prod\limits_{n =1}^\infty X_n$ converges in $R [\frac 1 p]$ and $ \fA : = Y*_\varphi \cA \in \rM _d ^h (\fS)$. 
 This is equivalent to that we will find another basis $e_1 , \dots, e_d $ of $R [\frac 1 p] \otimes \cM$ so that the $\fS$-submodule $\M '$ generated by $\{e_1, \dots , e_d\}$ is a Kisin module of height $h$. By \cite[Prop. 2.1]{BLL-semistable}, $\M'= \M(T')$ for $T'\subset V$ a $G_\infty$-stable $\cO_F$-lattice\footnote{We do not know if $T' = T$.}. In particular, let $\overline \M' : = \M' / \varpi \M'$ and $\overline V$ be the \emph{semi-simplification} of the reduction $ T / \varpi T$ then $\overline V$ is uniquely computed by $V^*_{\F} (\overline \M'[\frac 1 u])$ by \cite[Cor. 2.3.2]{BergdallLevin-BLZ}.   We will return the discussion of $V ^*_{\F}$ in more details in \S \ref{sec-Galoisrep}. 
\subsection{Initial matrices of Frobenius} 
By the discussion around equation \eqref{eqn-compute-phi}, we see that $$\cA =[1, E ^r , E ^s] A [1, \nu^{-r}, \nu^{-s} ]$$ with  matrices $A$ in Theorem \ref{them-classification-SD-lattices}. 

We need to simplify  $\cA$ so that $\det(\cA) = E^{r+ s}$. 
Set $z: =  \prod\limits_{i= 0}^{\infty} \varphi^{3i}(\nu^{r}\varphi^2 (\nu ^s)) $, $y = \varphi (\nu^s) \varphi ^2(z) $, and $x = \nu^s \varphi (z)$. Note that all $x,  y, z \in S ^\times$ as $\nu \in S^\times$. Set $X= [x, y, z ]  \in \GL_3 (S)$, we have $X *_\varphi$ type I matrix: 
 \begin{equation}\label{eqn-A}
 \cA= \begin{pmatrix}
 a x\varphi (x^{-1}) & b \nu ^{-r} x \varphi (y ^{-1}) & 1 \\  E^r & c E^{r} \nu^{-r}y \varphi (y ^{-1})  & 0 \\ 0 & E^s & 0 
 \end{pmatrix}
 \end{equation}
 Similarly, for type II matrix $A$, we may use $X*_\varphi [1, E ^r , E ^s] A [1, \nu^{-r}, \nu^{-s} ]$ with $X= [x' , y' , z' ]$ where $x'=\prod\limits_{i= 0}^{\infty} \varphi^{3i}(\nu^{r}\varphi (\nu ^s)) $, $y' = \nu ^s \varphi ^2 (x')$ and $z' = \varphi (x')$.  We obtain:  
  \begin{equation}\label{eqn-A-II}
 \cA= \begin{pmatrix}
 a x'\varphi ({x'}^{-1}) & 1 & c \nu^{-s} x'  \varphi({z'}^{-1}) \\  b E^r y ' \varphi  ({x'}^{-1}) &  0   & E ^r  \\ E^s & 0& 0 
 \end{pmatrix}
 \end{equation}

 Write $\gamma_i (E): = (\frac{E^p} p )^i. $For any $w\in S$, we can uniquely write $w =  \sum \limits_{i = 0}^\infty \fw_i \gamma_i (E) $ with $\fw_i\in \cO[u]$ with maximal degree $p -1$. 
We have the following estimate: 
\begin{lemma}\label{lem-estimate}
\begin{enumerate}
	\item $x\varphi (x^{-1})= \sum \limits_{i = 0}^\infty \fa_i \gamma_i (E)  $ with $\fa_0 \equiv 1 \mod p$ and $\fa_1 \equiv s \mod p$.  
	\item $\nu ^{-r} x \varphi (y ^{-1}) = \nu^{s-r}\varphi (z y^{-1})=  \sum \limits_{i = 0}^\infty \fb_i\gamma_i (E) $ with $\fb_0 \equiv 1 \mod p$ and $\fb_1 \equiv s-r \mod p$.
	\item $ \nu^{-r}y \varphi (y ^{-1}) =  \sum \limits_{i = 0}^\infty \fc_i \gamma_i (E) $ with $\fc_0 \equiv 1 \mod p$. 
		\item $x'\varphi ({x'}^{-1})= \sum \limits_{i = 0}^\infty \fa'_i \gamma_i (E)  $ with $\fa'_0 \equiv 1 \mod p$ and $\fa'_1 \equiv r \mod p$.
		\item $  \nu ^{-s} x'  \varphi({z'}^{-1}) = \nu^{r- s }\varphi (y' {z'}^{-1})=  \sum \limits_{i = 0}^\infty \fb'_i\gamma_i (E) $ with $\fb_0 \equiv 1 \mod p$ and $\fb'_1 \equiv r-s \mod p$.
	\item $ y ' \varphi  ({x'}^{-1})  =  \sum \limits_{i = 0}^\infty \fc'_i \gamma_i (E) $ with $\fc'_0 \equiv 1 \mod p$. 
\end{enumerate}

\end{lemma}
\begin{proof} We will only prove (1), while other statements can be proved similarly. We note that $x \varphi (x^{-1}) = \nu ^s \alpha$ with $\alpha = 1 + \sum _{i \geq 1} a_i ( \frac{u ^{p^2}}{p})^i. $ Using that $\nu = 1 + \frac{u ^p}{p}$, we can easily prove (1) by binomial expansion. 
\end{proof}

The following Proposition is key point we compute the reduction. Recall  $R : = \cO[\![ u,  \frac{\varpi E}{p}]\!]$. 

\begin{prop}\label{prop-cut-tail}Assume that  $p \geq 3$ and $A \in \rM^h_d (S)$. Suppose that $A = A_1 + C_1$ with $A_1 \in  \rM^h_d (\fS)$ and $C_1 \in \rM_d (\varpi \Fil ^{h_1} S)$ with $h_1\geq   \frac{h}{1- 2 p ^{-1}} $.  Then there exists $A_n \in \rM^h_d (\fS)$,  $X_n \in \GL _d (R)$ and $h_n \geq h_1$ so that 
	\begin{enumerate}
		\item $X_n *_\varphi   A = A_{n +1} + C_{n+1} $ with $A_{n +1} \in \rM^h_d (\fS) $ and $C_{n +1} \in \rM_{d} (\varpi\Fil ^{h_{n +1}} S)$. 
		\item Let $n \to \infty $. Then $h _n \to + \infty$,  $A_n$ converges an $\mathfrak A \in \rM_d^h (\fS)$ and the sequence $X_{n}$ converges in $\GL_d(R)$. 
		\item $ \mathfrak A \equiv A_1 \mod \varpi$. 
	\end{enumerate}	   
\end{prop}
\begin{proof} We make induction on $n$. Let $A=A_0$, $X_0= I_d$. Then the case for $n = 0$ is clear. Now assume the statement is valid for $n = m-1$. Let us construct $A_{m +1}$ and $C_{m +1}$. Since $C_{m} \in \rM_d (\varpi \Fil ^ {h_m} S)$, we can write $C_m =  C'_m E^h=  C'_m B_m A_m$. Therefore   
	$$X_{m -1} *_\varphi   A= A_m + C_m = A_m + C'_m B_m A_m = (I_d + C'_m B_m) A_m.$$  
	Set $X_m = (I_d+ C'_m B_m)^{-1} X_{m -1}$. Let us accept that $X_m \in \GL_d (R)$ and postpone the proof later. 
	Then 
	\[  X_m *_\varphi  A  = A_m (I _d + \varphi (C'_m B_m))\]
	Let  $x$ be an entry of $C'_m$. It is clear that $x = \sum\limits_ {i \geq h_m } a_i \frac{ E^{i - h}}{p ^{\lfloor \frac i p \rfloor}}$ with $a _i \in \varpi \cO$.  As $ \varphi (E) = p\nu$ with $\nu \in S^\times$, we have $\varphi (x) \in \varpi p ^{l _m} S$ where $l_m =  h_m -h - \lfloor \frac {h_m}{p} \rfloor $. Since any $ f \in S $ can be written as $\sum\limits _{i \geq 0} b_i \frac {E^i}{p^{\lfloor \frac i p\rfloor}}$ with $b_i \in W(k)[u]$, we have $\varphi (x)= y + z$ with $ y \in \varpi \fS \cap p ^{l _m} S$ and $z \in \varpi \Fil ^ {h _{m +1}} S$ with $h _{m +1} = p (l_m +1)$. Consequently, $\varphi (C'_m B_m) = D_{m +1} + D' _{m +1}$ with $D_{m +1} \in \rM _d (\varpi  \fS \cap p ^{l_m} S)$ and $D'_{m +1} \in \rM_d (\varpi \Fil ^{h_{m +1}} S)$. 
	We then set $A_{m +1} = A_m (I_d + D_{m +1})$ and $C_{m +1} = A_m D' _{m +1 }$. Since $D_{m +1} \in \rM _d (\varpi \fS \cap p ^{l_m} S)$, $ (I_d + D_{m +1}) ^{-1} = \sum_{i \geq 0} (-D_{m +1}) ^i\in \GL_d(\fS )$, we conclude that $A_{m +1} \in \rM ^h _d (\fS)$,  $A_{m +1}\equiv A_m \mod \varpi$ inside $\rM _d (\fS)$, and $A_{m +1}- A_m\in \rM_d  (p ^{l_m} S) $. So if $h_m \to + \infty$ then $l _m \to + \infty$, and consequently, $A_m$ converges to  $\mathfrak A \in \rM ^h _d(\fS)$ and $\mathfrak A\equiv A _1 \mod \varpi$. 
	
	Now  we must show that $h_{m +1} > h_m$, which is  equivalent to 
	$ h_m (1 -p ^{-1 }) - \lfloor \frac {h_m}{p} \rfloor +1 > h$. We have $h_1 (1 -p ^{-1 }) - \lfloor \frac {h_1}{p} \rfloor +1 >  h _1 (1 - p ^ {-1}- p ^{-1}) \geq h$. So 
	$h _2> h_1 $. For $m >  2$, we have $ h_m (1 -p ^{-1 }) - \lfloor \frac {h_m}{p} \rfloor \geq h_m ( 1 - 2 p ^{-1}) \geq h_2 (1- 2 p ^{-1}) > h$ as required. 
	
	Finally we need to check that $X_n\in \GL_d (R)$ and $X_n$ converges in $\GL_d(R)$. First note that $C_{m} \in \rM_d (\varpi \Fil ^ {h_m} S)$, each entries of  $C'_m B_m $ has the shape  $ x = \sum\limits_ {i \geq h_m } a_i \frac{ E^{i - h}}{p ^{\lfloor \frac i p \rfloor}}$ with $a _i \in \varpi \cO$. Since $h_m \geq h_1 > h/ (1- p ^{-1})$, we see $i -h \geq \lfloor \frac i p \rfloor$.  Hence $x\in R$ and $(I_d  + C'_m B_m) ^{-1} \in \GL_d(R)$. Since $h_m \to + \infty$, we see that $X_n : = \prod\limits_{i =0}^n (I_d + C'_i  B_i)^{-1}$ converges in $\GL_d (R)$.  
\end{proof}
\subsection{Algorithm to compute $\overline \M$} \label{subsec-algorithm} We set $I : = \varpi p S + \varpi \Fil ^{2p} S + E \Fil ^{2p} S$ and $J : = (\varpi, E , \frac{E^2}{p}) \Fil ^{2p}S$.  Note that $I \subset S$ is an ideal but $J\subset S[\frac 1 p]$ is only an $S$-module. 
 Let $x\in I$ then $1+x\in S ^\times$ and $ (1+x) ^{-1}-1 \in I$.

We will apply the following  operations $X_n$ to matrices $\cA_0 = \cA$ in \eqref{eqn-A} and \eqref{eqn-A-II}  so that 
 $\cA_{n +1} = X_n *_\varphi \cA_{n}$  and  
\begin{enumerate}
\item $X_n $ is either $[ x, y,  z]\in \GL_d (F)$  or $X_n \in \GL_d (S[\frac 1 p])$ with $\det (X_n)=1$.  
\item After finite many steps,  $\cA_m = \mathfrak A_0 + C_0$ so  that  $\mathfrak A_0 \in \GL_3 (\fS)$ and $ C_0 \in \rM_3 (I + J)$. 
\item The second row of $\fA_0$  and $C_0$ are $E^r a_{2j}, E ^r b_{2j}$ with  $a_{2j}\in \fS$ and $b_{2j} \in I+  \Fil^{2p} S $ for each $j$ respectively. 
\end{enumerate}
\begin{prop}\label{prop-algorithm-OK}
	Notations as the above. There exists an invertible matrix $Y \in \GL_3 (R[1/p])$ and $\fA \in \rM_3^h (\fS)$ so that $Y *_\varphi \cA = \fA$ and $\fA\equiv \fA_0 \mod \varpi$.  
\end{prop} 
\begin{proof}
	For any $x\in I$, we can write $x = x_0 + x_1$ with $x \in \varpi \fS$ and $x _1 \in J  = (\varpi , E, \frac{E^2}{p})\Fil ^{2p} S$. If we split $C_0$ in such way then $\cA_m = \fA' + \tilde C$ so that  
	$\tilde C \in \rM_3 (J)$ with second rows entries in $E^r \Fil^{2p}S$ and $\fA_0 \equiv \fA' \mod \varpi$. Now we claim that $\fA' \in \rM _3^h (\fS)$ where $h = s \leq 2p -4$. 
	
	To show the claim, we first note that (1) in the above implies that $\det (\cA_m) = E ^{r+s}$. Since the second row of and $\fA'$ and $\tilde C$ has factor $E^r$, we easily see that 
	$\det (\fA') \equiv E^{r+s} \mod E^r \Fil ^{2p} S$. Then $\det (\fA') - E ^{r+s} \in E^r \Fil^{2p}S \cap \fS = E^{r+ 2p} \fS$.   Since $s < 2p$, we see that $\det (\fA') = E ^{r+s} \mu$ with $\mu \in \fS^\times$. Since $\cA \in \rM^h_3 (S)$, $\cA_m \in \rM^h_3 (S[\frac{1}{p}]) $. Then there exists a $B \in \rM_3 (S [\frac 1 p])$ so that $\cA_m B = E^h I _3$. So 
	$\fA ' B = E^h I _3 + D'$ so that $D'  \in \rM_3 (\Fil^{2p} S [\frac 1 p])$. Write $D'= E ^{2p } D $. We see that $\fA' B' = E^h I_3$ with $B ' = B  (I + E^{2p - h} D)^{-1}$. Note that entries of $B'$ are in $S_E : = F[\![E]\!]$. Since $\det (\fA') = E^{r+ s}\mu$. There exists a matrix $B''\in \rM_3 (\fS)$ so that $B'' \fA' = E^{r+s} I_3$. Therefore, $E^{r+s} B '  = E^h B''$. Since all entries of $E ^h B''$ are in $\fS$. This forces all entries of $B'$ are in $\fS$ because $\fS\cap E ^n S_E = E^n \fS$. Therefore, $\fA' $ is in $\rM^h_3 (\fS)$ as claimed.  
	
	We can not use Proposition \ref{prop-cut-tail} now as $J$ is not inside $\varpi \Fil^{2p} S$. Since $\fA'$ has height $h$, we can decompose $\tilde C=  C' E^h = C' B \fA'$. Let $X= (I_3 + C' B  )^{-1}$. As in the proof of Proposition \ref{prop-cut-tail}, we have 
	$X *_\varphi \cA_m = \fA' (I + \varphi (C' B))$. Since $\tilde C \in \rM_3 (J)$, we see each entry of $C'$ has the form $\sum\limits_{j \geq 2p} a_j\frac {E^{j -h}}{p ^{\lfloor \frac j p \rfloor}}$ with $a_j \in \varpi \fS + E \fS + \frac{E^2}{p} \fS$. Use $\varphi (E) = p \nu $ with $\nu \in S^\times$ and $h \leq 2p-4$, we see that $\varphi (C') \in \rM _3 (\varpi p ^2 S )$. Therefore $\varphi (C')=  C'' + D''$ with $C'' \in \rM _3 (\varpi \fS)$ and $D'' \in \rM _3 (\varpi \Fil ^{2p} S)$. Set $A_1 = \fA'(I_3 + C'' \varphi (B) )$ and $C_1 = \fA' D'' \varphi (B)$, and apply Proposition \ref{prop-cut-tail} to $A_1$ and $C_1$. Then $\fA \equiv \fA' \equiv \fA_0 \mod \varpi $ is required. 	
\end{proof}

We need to enlarge the ring of coefficient to discuss deformation theory involved in Theorem \ref{thm-deformation}. Let $(\Lambda, \fm)$ be a complete noetherian local $\cO_F$-domain, $\fS _{\Lambda}: = \Lambda[\![u ]\!]$, $S_\Lambda : = \fS_\Lambda [\![\frac {E^p}{p}]\!]$ and $R_\Lambda$ be the $(\fm, u)$-completion of $\Lambda [u , \frac E p ]$.  
 We extends $\varphi$ on $\fS $ to $\fS_\Lambda$ and $R_\Lambda$  $\Lambda$-linearly. We define \emph{Kisin module over $\fS_\Lambda$ of height $h$}  to be a finite free $\fS_\Lambda$-module $\M_{\Lambda}$ with a $\varphi _{\fS_\Lambda}$-semi-linear map $\varphi_{\M_{\Lambda}} : \M_{\Lambda} \to \M _{\Lambda} $ so that  $\coker (1 \otimes \varphi _{\M_\Lambda})$ is killed by $E^h$, where $1 \otimes \varphi _{\M_\Lambda} : \fS_{\Lambda} \otimes_{\varphi, \fS_\Lambda } \M_{\Lambda} \to \M_{\Lambda}$ is the linearization of $\varphi_{\M_\Lambda}$. 
 Now we modify the requirements of algorithm before Proposition \ref{prop-algorithm-OK} as following: set $I_{\Lambda} = \fm pS_{\Lambda} + \fm \Fil ^{2p}S_{\Lambda} + E \Fil ^{2p} S_{\Lambda} $ and $J_{\Lambda}: = (\fm , E , \frac{E^2} {p}) \Fil ^{2p} S_{\Lambda}$. Now replace $I$, $J$, $F$ by $I _\Lambda$, $J _{\Lambda}$ , $\Lambda[\frac 1 \fm]$ respectively, we have the Proposition \ref{prop-algorithm-OK} with coefficients in $\Lambda$.  
 \begin{prop}\label{prop-algorithm-lambda}
 	Using modified notations as the above. There exists an invertible matrix $Y \in \GL_3 (R[\frac 1 \fm])$ and $\fA \in \rM_3^h (\fS_{\Lambda})$ so that $Y *_\varphi \cA = \fA$ and $\fA\equiv \fA_0 \mod \fm$.  
 \end{prop} 
\begin{proof} The proof is the same as that of Proposition \ref{prop-algorithm-OK} (which essentially use Proposition \ref{prop-cut-tail}) by suitable replacement, like $S$ replaced by $S_\Lambda$. Also we need to replace $\varpi$ by $\fm$ everywhere.  Proposition \ref{prop-cut-tail} in this version (the same proof holds after the above replacement) is the following: Assume that  $A \in \rM^h_d (S_\Lambda)$. Suppose that $A = A_1 + C_1$ with $A_1 \in  \rM^h_d (\fS_\Lambda)$ and $C_1 \in \rM_d (\fm  \Fil ^{h_1} S_\Lambda)$ with $h_1\geq   \frac{h}{1- 2 p ^{-1}} $.  Then there exists $A_n \in \rM^h_d (\fS_\Lambda)$,  $X_n \in \GL _d (R_\Lambda)$ and $h_n \geq h_1$ so that 
\begin{enumerate}
\item $X_n *_\varphi   A = A_{n +1} + C_{n+1} $ with $A_{n +1} \in \rM^h_d (\fS_{\Lambda}) $ and $C_{n +1} \in \rM_{d} (\fm \Fil ^{h_{n +1}} S_\Lambda)$. 
\item Let $n \to \infty $. Then $h _n \to + \infty$,  $A_n$ converges an $\mathfrak A \in \rM_d^h (\fS_\Lambda)$ and the sequence $X_{n}$ converges in $\GL_d(R_\Lambda)$. 
\item $ \mathfrak A \equiv A_1 \mod \fm$. 
\end{enumerate}	   
\end{proof}
 

\section{Row operations and computations}	
Now we carry out the algorithm in \S \ref{subsec-algorithm} to compute matrices $\cA$ in \eqref{eqn-A} and  \eqref{eqn-A-II}. Our goal is to find  $X_n$ and $\cA_m= \fA_0+ C_0$ explicitly in each case. Then we will find $\bar \cA = \fA_0 \mod \varpi$ explicitly, which corresponds to the  reduction $\overline \M' $ of a Kisin module $\M' = \M (T')$ by the discussion in the end of \S \ref{subsec-summary}. Recall that $I : = \varpi p S + \varpi \Fil ^{2p} S + E \Fil ^{2p} S$ and $\varphi(E)= p \nu$. In the following, we always reserve  $\eps_{ij}$,  $\eps'_{ij }$ when they are in $I$. For each steps of operations, if $\eps_{ij}$ is changed (updated) but still in $I$ then we always use $\eps'_{ij}$ to replace $\eps_{ij}$. 

\subsection{Type I and III matrices}	Now we first consider matrix $\cA$ in \eqref{eqn-A} with possibly $c\in \cO^\times$ (which allows type III by \S \ref{rmk-III-I}).   
We first  have $$\begin{pmatrix}
1  &  - a (\sum\limits_{i \geq 1} \fa_i \frac{ E^{p i -r}}{p ^i})  & 0 \\  0 & 1  & 0 \\ 0 & 0 & 1 
\end{pmatrix} *_\varphi \cA  = \begin{pmatrix}
a \fa_0    &   \fd  & 1 \\  E^r &  c E^{r} \nu^{-r}y \varphi (y ^{-1}) + E^r \eps_{22}  & 0 \\ 0 & E^s  & 0  
\end{pmatrix}$$ where $\eps_{22} = a \varphi ( \sum _{i \geq 1}\fa_i \frac{ E^{p i -r}}{p ^i})\in I $,  
and 
\[\fd :  = b \sum_{i \geq 0} \fb_i \gamma_i (E) - ac \left (\sum_{i\geq 1 }\fa_i \gamma _i (E) \right) \left (\sum_{i \geq 0} \fc_i \gamma_i (E)\right )  +  a^2\fa_0  \varphi \left ( \sum _{i \geq 1}\fa_i \frac{ E^{p i -r}}{p ^i}\right )  = 
b \fb_0 + \sum_{i \geq 1} \fd_i \gamma_i (E) \]
Using $\varphi (E) = p \nu$, if we write $\fd_i = \sum\limits_{j = 0}^{p-1}\fd_{ij} u ^j $ with $\fd_{ij} \in \cO$, $\lambda: = \min \{v_p (b), v_p (ac), v_p (a^2)\}$ then 
it is easy to check via Lemma \ref{lem-estimate} that 
\begin{equation}\label{eqn-b'}\fd_{10} \equiv (s-r)b - sac \mod \varpi p; \ \   v_p (\fd_{1j})  \geq 1 + \lambda, \  1 \leq j \leq p-1 \text{ and } v_p(\fd_{ij})\geq \lambda,\ i \geq 2.   \end{equation}  Now update $\cA$ and apply $\begin{pmatrix}
	1  &  0   & 0 \\  0 & 1  & - c \sum\limits_{i \geq 1} \fc _i \frac{E^{ip +r -s}}{p ^i} \\ 0 & 0 & 1 
\end{pmatrix} *_\varphi \cA $, we obtain $\begin{pmatrix}
	a \fa_0    &   \fd  & 1 + \eps_{13} \\  E^r &  c \fc_0 E^{r}   + E^r \eps_{22}  & E^r c\delta _{23} \\ 0 & E^s  & E^s c\delta_{33}  
\end{pmatrix}$
where $\delta_{j3} =  \alpha_j \sum\limits_{i \geq 1}  \varphi (\fc _i \frac{E^{ip +r -s}}{p ^i})\in pS $ where $\alpha_j$ from second column for $j = 2, 3 $, and $\eps_{13} =  \fd  \sum\limits_{i \geq 1}  \varphi (\fc _i \frac{E^{ip +r -s}}{p ^i})\in \varpi pS \subset I. $ 
Set $$b': = (s-r)b -sac . $$
We have following cases to discuss. 
\subsubsection{$v_p (b')> 1$.}\label{subsubsec-first-case} Then $\bar {\cA}  = \begin{pmatrix}0 & 0 & 1 \\ u ^r & 0 & 0 \\ 0 & u ^s & 0 \end{pmatrix}$ for type I matrix because $\fd_1 \gamma_1(E) \in \varpi \fS$ and $c \delta_{i 3}\in I$ in this situation. For type III matrix, $c \in \cO^\times$, write $\bar c = c\mod \varpi$.  
We have $c E^s\delta_{33}\in I$, also  $c \delta_{23}= \alpha E^p + \eta$ with $\alpha \in \cO [u]$ and $\eta\in p \fS + p \Fil ^{2p }S\subset p \fS + J$. 
 Therefore  $\bar \cA = \begin{pmatrix}0 & 0 & 1 \\ u ^r & \bar c u ^r &  \bar \alpha u ^{r+p} \\ 0 & u ^s & 0 \end{pmatrix}$ where $\bar \alpha = \alpha \mod \varpi$. 
	
\subsubsection{$v_p (b')=1 $.} \label{subsubsec-2nd-case} In this case,  let $\bar b' = \frac{b'}{p}\mod \varpi$. For similar argument as the above. For type I,   we have 	
 $\bar {\cA}  = \begin{pmatrix}0 & \bar b '  u ^p & 1 \\ u ^r & 0 & 0 \\ 0 & u ^s & 0 \end{pmatrix}$.   For type III, we have 
  $\bar {\cA}  = \begin{pmatrix}0 & \bar b '  u ^p & 1 \\ u ^r & \bar c u ^r & \bar \alpha u ^{r+p} \\ 0 & u ^s & 0 \end{pmatrix}$. 
 
 \subsubsection{$v_p (b ') < 1$ and $v_p (a)\leq v_p (b)$.}\label{subsubsec-I-1}
 Now the real challenging case is when $v_p(b ') < 1$, which we will assume from right now. Write $d : = \fd_{10} $, first note in this case $v_p (b') = v_p(d)$ by \eqref{eqn-b'}. 
 Our next row operation is to kill $E^s$-term.  Consider 
 $\begin{pmatrix}
 1  &  0   & 0 \\  0 & 1  & 0 \\ -\frac{p}{d} E^{s-p}  & 0 & 1 
 \end{pmatrix} *_\varphi \cA $, we obtain 
  \begin{equation}\label{eqn-Matrix}
  	\cA = \begin{pmatrix}
  	a \fa_0  + \eps_{11}  &   \fd  & 1 + \eps_{13} \\  E^r(1 + \eps_{21} )  &  c \fc_0 E^{r}   + E^r \eps_{22}  & E^r c \delta_{23} \\ -\frac{pa \fa_0}{d} E^{s-p} + \frac{p^2}{d} E^{s-p}\eps_{31} & -\frac{pb\fb_0}{d} E^{s-p}-\frac{p}{d} E^{s-p} \sum\limits_{i \geq 1} \fd'_{i}\gamma_i(E)     & - \frac{p}{d} E^{s-p} (1 + \eps'_{33})  
  	\end{pmatrix}
  \end{equation}
where $\fd'_i= \fd_i, i \geq 2$,  $\fd'_1 = \fd_1 - \fd_{10}$ and $- \frac{p}{d} E^{s-p} (1 + \eps'_{33})= - \frac{p}{d} E^{s-p} (1 + \eps_{13}) + E^s c \delta_{33}$. Since $E^s \delta _{33} \in  E ^s p  S = E^{s-p} p \gamma_1 (E) p  S$, we see $\eps'_{33} \in p \varpi S$; All terms of $\eps_{i 1}$ comes from column operation which add $3$rd columns multiplying 
$\frac p d \varphi (E ^{s-p})$ to the first column. Note that $\varphi(E^{s-p})\in p ^ 2 S$, we conclude that all $\eps_{i 1} \in p \varpi S$.

Now first assume that  $v_p (a)\leq v_p (b)$. Update $\eps'_{33}$ by $\eps_{33}$.   Apply  $\begin{pmatrix}
1  &  0   & 0 \\  0 & 1  & 0 \\ 0   & 0 &  \frac{d}{pa} 
\end{pmatrix} *_\varphi \cA $, we update  $$\cA = \begin{pmatrix}
a \fa_0    &   \fd & \frac{pa }{d} (1 + \eps_{13})  \\  E^r(1+ \eps_{21}) &  c \fc_0 E^{r}   + E^r \eps_{22}  & \frac{pa}{d} E^r c\delta _{23} \\ -{\fa_0} E^{s-p} (1 + \eps'_{31}) & -\frac{b\fb_0}{a} E^{s-p}-\frac{1}{a} E^{s-p} \sum\limits_{i \geq 1} \fd'_{i}\gamma_i(E)   & - \frac{p}{d} E^{s-p} ( 1+ \eps_{33})  
\end{pmatrix}$$
Here $\eps'_{31} = \frac{p}{a} \eps_{31} \in \varpi p S $ because $v_p (a)\leq v_p (b)$ implies that $v_p (a)\leq v_p (b')< 1$. Since $v_p (a)\leq v_p (b)$, \eqref{eqn-b'} implies that $\frac{1}{a} E^{s-p} \sum\limits_{i \geq 1} \fd'_{i}\gamma_i(E) \in S$ and $\frac 1 a \fd'_1 \gamma _1(E) \in \fS$. 
Now update $\eps'_{31}$ by $\eps_{31}$ and set $ \fd' : = b \fb_0+  \sum\limits_{i \geq 1} \fd'_{i}\gamma_i(E)$. 
Let us  consider 
$\begin{pmatrix} 1 &  0 & 0 \\ 0 & 1 & (\fa_0(1+ \eps_{31}))^{-1} E^{r+p -s} (1+ \eps_{21})\\ 0 & 0 & 1\end{pmatrix}*_\varphi \cA$, we have 
 $$\cA: = \begin{pmatrix}
a \fa_0    &   \fd & \frac{pa }{d} (1 + \eps'_{13})  \\  0  & E^{r}( c \fc_0  - \frac{1}{\fa_0 a} \fd'  + \eps'_{22} ) & -\frac{p}{d} \fa_0^{-1}  E^r(1+  \eps'_{23})  \\ -{\fa_0} E^{s-p} (1 + \eps_{31}) & -\frac{1}{a} E^{s-p} \fd'    & - \frac{p}{d} E^{s-p} ( 1+ \eps'_{33})  
\end{pmatrix}$$
 where $E^{r}( c \fc_0  - \frac{1}{\fa_0 a} \fd'  + \eps'_{22}) = c \fc_0 E^{r}   + E^r\eps_{22}  - (\fa_0(1+ \eps_{31}))^{-1} E^{r+p -s} (1+ \eps_{21})\frac{1}{a} E^{s-p} \fd' $; and $\eps'_{i 3}$ comes from multiplying 2nd columns $\varphi((\fa_0(1+ \eps_{31}))^{-1} E^{r+p -s} (1+ \eps_{21}))$ to the 3rd column (note that $23$-entry both comes from row operation and column operation. Since $a \delta_{23}\in \varpi pS$, $a \delta_{23}$ can be added into a part of $\eps'_{23}$). 
 Using $r+p -s \geq 2$, $\varphi (E)= p\nu $ and $v_p (a) \leq v_p (b) \leq v_p (d)$, we see that all $\eps'_{ij}\in \varpi p S$. 
Now rewrite $\eps'_{ij}=\eps_{ij} $ and 
make $\begin{pmatrix} 1 &  0 & 0 \\ 0 & \frac{d}{p} & 0 \\ 0 & 0 & 1\end{pmatrix}*_\varphi \cA$,  we get 
$$\cA: = \begin{pmatrix}
	a \fa_0    &   \frac{p}{d}\fd & \frac{pa }{d} (1 + \eps_{13})  \\  0  & E^{r}( c \fc_0  - \frac{1}{\fa_0 a} \fd'  + \eps_{22})  & -\fa_0^{-1}  E^r(1+  \eps_{23})  \\ -{\fa_0} E^{s-p} (1 + \eps_{31}) & -\frac{p}{da} E^{s-p} \fd'    & - \frac{p}{d} E^{s-p} ( 1+ \eps_{33})  
\end{pmatrix}$$
Note that $\fd_{10}= d$. So $\frac{p}{d}\fd = b \frac p d \fb_0 + E^p + \tilde \eps_{12}$ with $\tilde \eps_{12} \in \varpi \fS +  I$. Since $\frac{\fd'_0}{a}\equiv \frac b a  \mod p$,  $\frac 1 a \fd'_1 \gamma _1(E) \in \fS$ and $\frac 1 a \sum_{i \geq 2} \fd_i \gamma _i (E)\in \Fil^{2p} S$ by \eqref{eqn-b'}, we have $-\frac{p}{da} E^{s-p} \fd' \in\varpi \fS+  I$ and rewrite 
 $E^{r}( c \fc_0  - \frac{1}{\fa_0 a} \fd') = E^r (c\fc_0 -\frac{1}{\fa_0 a} (\fd_0' + \fd'_1 \gamma_1(E))) + E^r \eps''_{22}$ with $\eps''_{22} \in  \Fil ^{2p} S$. Write $ \bar  \beta u ^p= \frac 1 a \fd'_1 \gamma _1(E) \mod \varpi $. Note that $\bar \beta = 0$ if $v_p (b) > v_p (a)$ and $\fa_0 \equiv 1 \mod p$. Now we have two cases to discuss:
 

Case (1): if $v_p (a)= v_p (b)$. Let $\overline{a/b}: = a/b \mod \varpi$. Then $\bar \cA = \begin{pmatrix}0 &   u ^p & 0  \\ 0  &  -\overline{a/b} u ^r - \bar \beta u ^{r+p}  & u ^r \\ u ^{s-p}  & 0 & 0 \end{pmatrix}$ when $A$ is type I. If $A$ is type III, let $\bar \alpha = c - \frac b  a \mod \varpi$ then $\bar \cA = \begin{pmatrix}0 &   u ^p & 0 \\ 0 & \bar \alpha u ^r  - \bar \beta u ^{r+p} & u ^r  \\ u ^{s-p} & 0 & 0 \end{pmatrix}$. 

Case (2): if $v_p (a) < v_p (b)$.  Then $\bar \cA = \begin{pmatrix}0 &   u ^p & 0 \\ 0  &  0  & u ^r \\ u ^{s-p} & 0 & 0 \end{pmatrix}$ for type I and  $\bar \cA = \begin{pmatrix}0 &   u ^p & 0 \\ 0  &  \bar c u ^r  & u ^r \\ u ^{s-p} & 0 & 0 \end{pmatrix}$ for type III with $\bar c = c \mod \varpi$. 

\subsubsection{$v_p (b') < 1$ and $v_p (a)> v_p (b)$.}\label{subsubsec-hardest}The most difficult case is that $v_p (b)< v_p (a)$ in which case $v_p (b')= v_p (b)= v_p (d)$. We return to the matrix \eqref{eqn-Matrix},   apply  $\begin{pmatrix}
1  &  0   & 0 \\  0 & 1  & 0 \\ 0   & 0 &  \frac{d}{pb} 
\end{pmatrix} *_\varphi \cA $, then  \begin{equation}\label{eqn-hardcase}\cA= \begin{pmatrix}
a \fa_0    &   \fd & \frac{pb}{d} (1 + \eps_{13})  \\  E^r(1+ \eps_{21}) &  c \fc_0 E^{r}   + E^r \eps_{22}  & \frac{pb}{d} E^r  c\delta_{23} \\  E^{s-p} (-\frac {\fa_0a}{b} + \frac{p}{b}\eps_{31}) & -\fb_0E^{s-p}-\frac{1}{b} E^{s-p} \sum\limits_{i \geq 1} \fd'_{i}\gamma_i(E)   & - \frac{p}{d} E^{s-p} ( 1+ \eps_{33})  
\end{pmatrix}\end{equation} and make $\begin{pmatrix}
1  &  0   & \frac{d}{\fb_0p}  E^{2p -s} \\  0 & 1  & 0 \\ 0   & 0 &  1 
\end{pmatrix} *_\varphi \cA$, we obtain 
$$\cA= \begin{pmatrix}
a \fa_0  + E ^p\frac{d}{\fb_0p} (-\frac {\fa_0a}{b} + \frac{p}{b}\eps_{31})      &   \fd - \frac{d}{p }E^{p}-\frac{d}{pb\fb_0} E^{p} \sum\limits_{i \geq 1} \fd'_{i}\gamma_i(E) &  - {\fb_0}^{-1} E^{p} + \frac{pb}{d} + \eps'_{13}  \\  E^r(1+ \eps_{21}) &  c \fc_0 E^{r}   + E^r \eps_{22}  & \frac{pb}{d} E^r c\delta'_{23} \\  E^{s-p} (-\frac {\fa_0a}{b} + \frac{p}{b}\eps_{31}) & -\fb_0E^{s-p}-\frac{1}{b} E^{s-p} \sum\limits_{i \geq 1} \fd'_{i}\gamma_i(E)   & - \frac{p}{d} E^{s-p} ( 1+ \eps'_{33})  
\end{pmatrix}$$ 
where $\eps'_{i 3}$ and $\delta'_{23}$ come from multiplying $\varphi (\frac{d}{\fb_0p}  E^{2p -s})$ to the last columns. Since $2p-s \geq 4$, we see that $\eps'_{i3}\in p\varpi S $ and $\delta'_{23}\in pS$. Note that 
$\frac{d}{pb\fb_0} E^{p} \sum\limits_{i \geq 1} \fd'_{i}\gamma_i(E)= \frac{d}{b\fb_0} \gamma _1 (E)  \sum\limits_{i \geq 1} \fd'_{i}\gamma_i(E)\in \varpi \Fil^{2p} S$ and $\fd - \frac{d}{p }E^{p} = \fd - \frac{\fd_{10}}{p }E^{p}= a_{12} + \delta_ {12}$ with $a_{12} \in \varpi \cO[u] $ and $\delta_{12} \in \varpi \Fil ^{2p}S$ (recall $d= \fd_{10}$). 
 By \eqref{eqn-b'},  $\frac 1 b \sum\limits_{i \geq 1} \fd'_{i}\gamma_i(E) = \frac{\fd'_1}{b}\gamma_1 (E) + \frac 1 b  \sum\limits_{i \geq 2} \fd'_{i}\gamma_i(E)$ with $\frac{\fd'_1}{b}\gamma_1 (E) \in E^p  \cO[u]$ and $\delta_{32}: = \frac 1 b  \sum\limits_{i \geq 2} \fd'_{i}\gamma_i(E)\in \Fil ^{2p} S$. Since $\fb_0 \equiv 1 \mod p$ by Lemma \ref{lem-estimate}, if we write 
$\fb_0+ \frac{1}{b} \sum\limits_{i \geq 1} \fd'_{i}\gamma_i(E) : = \nu _{32}$ then $\nu_{32} \in S ^\times$ and $\nu_{32}-1 \in (p, E^p) \fS + \Fil^{2p} S$.

In case that $v_p (a)>  1$,  inside  $11$-entry, we have $E ^p\frac{d}{\fb_0p} \frac {\fa_0a}{b}  \in \varpi \cO[u]$. So $\cA $ satisfies the requirements of \S \ref{subsec-algorithm}, we have 
$\bar \cA = \begin{pmatrix}0 &   0 & - u ^p  \\ u ^r   &  0  & 0  \\ 0  & -u ^{s-p}\bar \nu_{32} & 0 \end{pmatrix}$ for type I and  $\bar \cA = \begin{pmatrix}0 &   0 & - u ^p  \\ u ^r   &  \bar c u ^r  & 0  \\ 0  & -u ^{s-p}\bar \nu_{32} & 0 \end{pmatrix}$ for type III, where $\bar \nu_{32} \equiv 1 \mod u ^p\in \F [\![u]\!]^\times$.

So we can assume that $v_p (a)\leq 1$ from now on. To simplify notations, let $\wt a :  = \frac{a\fa_0}{b}$,   $d^{(1)}: = d \fb_0^{-1}$, $\mu_{11}: =1 - \frac {p}{a\fa_0} \eps_{31}$ and $\mu_{21}	: = 1 + \eps_{21}$.   Rewrite $$\fd - \frac{d}{p }E^{p}-\frac{d}{pb\fb_0} E^{p} \sum\limits_{i \geq 1} \fd'_{i}\gamma_i(E)= \sum_{i \geq 0} \fd''_i \gamma_i (E)= \fd^{(1)} +  \sum_{i \geq 2} \fd''_i \gamma_i (E): =\fd^{(1)} + \eps_{12}$$
where $\fd^{(1)} : = \fd''_0 + \fd''_1 \gamma_1(E) \in \varpi \cO [u]$ and $\eps _{12}: = \sum_{i \geq 2} \fd''_i \gamma_i (E) \in \varpi \Fil ^{2p} S$. Also set  $\wt c :  = c \fc_0$, $\eps_{23}: = b \delta_{23} \in I$ and update  $\eps'_{ij}$  to $\eps_{ij}$. 
 Now we simplify our matrix to the following form: 
\begin{equation}\label{eqn-matrix-(1)}
\cA= \begin{pmatrix}
a \fa_0  - E ^p\frac{d^{(1)}}{p} \wt a  \mu_{11}      &   \fd^{(1)} + \eps_{12} &  - {\fb_0}^{-1} E^{p} + \frac{pb}{d} + \eps_{13}  \\  E^r\mu _{21} &  E^{r} \wt c    +  E^{r} \eps_{22}  & \frac{p}{d} E^r \eps_{23} \\ -\wt a   E^{s-p} \mu _{11} & -E^{s-p}\nu_{32}    & - \frac{p}{d} E^{s-p} ( 1+ \eps_{33})  
\end{pmatrix}
\end{equation}
Now we kill $ E ^p\frac{d^{(1)}}{p} \wt a  \mu_{11} $ in 11-entry by applying 
$\begin{pmatrix}
1  &   E^{p -r} \mu^{-1}_{21}  \frac{d^{(1)}}{p} \wt a  \mu_{11}  & 0  \\  0 & 1  & 0 \\ 0   & 0 &  1 
\end{pmatrix} *_\varphi \cA$, and note that $\varphi (E^{p-r}) \in p ^ 2 S$ as $r \leq p -2$, we update $\cA$ by 
$$\cA= \begin{pmatrix}
a \fa_0      &   \fd^{(1)} +  E^{p} \mu^{-1}_{21} \mu_{11}  \frac{\wt c d^{(1)}\wt a }{p} +   \eps'_{12} &  - {\fb_0}^{-1} E^{p} + \frac{pb}{d} + \eps'_{13}  \\  E^r\mu _{21} &   E^{r}\wt c     + E^r \eps' _{22}  & \frac{p}{d} E^r \eps_{23} \\ -\wt a   E^{s-p} \mu _{11} & -E^{s-p}\nu'_{32}     & - \frac{p}{d} E^{s-p} ( 1+ \eps_{33})  
\end{pmatrix}, $$ 
here $\eps'_{13}$ is updated by adding  $E^{p -r} \mu^{-1}_{21}  \frac{d^{(1)}}{p} \wt a  \mu_{11}$ multiplying  2nd row to first row (note $v_p (d^{(1)})= v_p (d)$), and $\eps'_{22}, \mu  '_{32}$  is updated by  adding  $-\varphi (E^{p -r} \mu^{-1}_{21}  \frac{d^{(1)}}{p} \wt a  \mu_{11}) $ multiplying  1st column to 2nd column, $\eps'_{12}$ is added by both above  row operation (note that $E^r E ^{p-r} = E^p \in pS$) and column operation. Also $\nu'_{32}$ is updated so that $ \nu' _{32}-1 \in (p, E^p) \fS + \varpi pS + \Fil ^{2p} S$. 
Update the matrix $\cA$ by removing all $'$ and $\begin{pmatrix}
1  &   0 & E^{2p-s}\nu_{32}^{-1}\mu^{-1}_{21} \mu_{11}  \frac{\wt c d^{(1)}\wt a }{p}  \\  0 & 1  & 0 \\  0  & 0 &  1 
\end{pmatrix} *_\varphi \cA$, we get 
$$\cA= \begin{pmatrix}
a \fa_0    + E^p \mu'_{11} \frac{\wt c d^{(1)} \wt a ^2}{p}  &   \fd^{(1)} + \eps_{12} &    	-{\fb_0^{-1}}E^{p} + \frac{pb}{d} + \eps'_{13}  \\  E^r\mu _{21} &   E^{r}\wt c     + E^r \eps_{22}  & \frac{pb}{d} E^r\eps'_{23} \\ -\wt a   E^{s-p} \mu _{11} & -E^{s-p}\nu_{32}    & - \frac{p}{d} E^{s-p} ( 1+ \eps'_{33})  
\end{pmatrix},$$ 
where $ \mu'_{11}: = \nu_{32}^{-1}\mu^{-1}_{21} \mu_{11}$, and $\eps'_{23}$ and $\eps'_{33}$ is updated from column operation by using $\varphi (E^{2p-s})\in p ^4 S$, $\eps'_{13}$ is updated from both row operation  (using $v_p(d^{(1)})= v_p (d)$) and column operation.  Compare the above matrix to that in \eqref{eqn-matrix-(1)}, except updating $  \eps_{ij}, \mu _{ij} $ and $\nu_{32}$, we see the essential difference is the (11)-entry that $d^{(1)}\wt a $ is replaced by 
$d ^{(1)} \wt c \wt a^2$. So we may repeat operations after \eqref{eqn-matrix-(1)}, then 11-entry is replaced by $d ^{(1)} \wt c^m  \wt a^{m +1}$, which is eventually go to $0$ when $m \to \infty$. Then we get the $\bar \cA = \begin{pmatrix}0 &   0 & - u ^p  \\ u ^r   &  0  & 0  \\ 0  & -u ^{s-p}\bar \nu_{32} & 0 \end{pmatrix}$ for Type I and $\bar \cA = \begin{pmatrix}0 &   0 & - u ^p  \\ u ^r   &  \bar c u ^r  & 0  \\ 0  & -u ^{s-p}\bar \nu_{32} & 0 \end{pmatrix}$ for Type III.

\subsection{Type II matrix}\label{subsec-TypeII} To simplify our notations,  our notations in this subsection are independent of those in the last subsections unless otherwise stated.
We rewrite matrix \eqref{eqn-A-II} to the following new form
  \begin{equation}\label{eqn-A-II-new}
\cA= \begin{pmatrix}
a \fa'  & 1 & c \fc'  \\  b E^r \fb '  &  0   & E ^r  \\ E^s & 0& 0 
\end{pmatrix}
\end{equation}
Set $b': = ra + (s-r)bc $. Here we consider two situations of $a, b,c$: 
\begin{enumerate}[label=(\roman*)]
	\item $a, b  , c \in \varpi \cO_F$; 
	\item  \label{situtaion-2}$a, b \in \varpi \cO_F, c \in F$ satisfying $p c \in \varpi \cO_F$ and $v_p (b')< 1$, with two cases: 
	\begin{enumerate}
\item $ v_p (b) < \min \{v_p (a), 1 \}$ and $v_p (b  c^2)> 0$ 
\item $v_p (a )< \min \{v_p (b), v_p (bc), v_p (pc)\}$ and $v_p(ac)> 0$. 
	\end{enumerate}
\end{enumerate}
Note in the Situation \ref{situtaion-2}, we always have $v_p (ac), v_p (b c), v_p (b c^2), v_p (b'c), v_p (pc / b')> 0 $. 	
	

We first  have $\begin{pmatrix}
1  &  - c (\sum\limits_{i \geq 1} \fc'_i \frac{ E^{p i -r}}{p ^i})  & 0 \\  0 & 1  & 0 \\ 0 & 0 & 1 
\end{pmatrix} *_\varphi \cA  = \begin{pmatrix}
 \fd    &   1 + \eps_{12}  & c \fc'_0 \\ b E^r \fb' &   E^{r}  \eps_{22}  & E^r \\ E^s & E^s \wt \eps_{32}  & 0  
\end{pmatrix}$ where  $\eps_{i 2}$ comes from adding multiple $ c \varphi ( \sum \limits_{i \geq 1}\fc'_i \frac{ E^{p i -r}}{p ^i})$ of first columns to the second column,  
and 
\[\fd :  = a \sum_{i \geq 0} \fa'_i \gamma_i (E) - bc \left (\sum_{i\geq 0 }\fb'_i \gamma _i (E) \right) \left (\sum_{i \geq 1} \fc'_i \gamma_i (E)\right )   = 
a\fa'_0 + \sum_{i \geq 1} \fd_i \gamma_i (E). \]
Write $\fd_i = \sum\limits_{j = 0}^{p-1}\fd_{ij} u ^j $ with $\fd_{ij} \in \cO$, $\lambda: = \min \{v_p (bc), v_p (a)\}> 0$ then 
it is easy to check by Lemma \ref{lem-estimate} that 
\begin{equation}\label{eqn-b'-II}\fd_{10} \equiv ra  + (s-r)bc= b'  \mod \varpi p; \ \   v_p (\fd_{1j})  \geq 1 + \lambda, \  1 \leq j \leq p-1 \text{ and } v_p(\fd_{ij})\geq \lambda,\ i \geq 2.   \end{equation}  
 It is not hard to check that $\eps_{i2}\in \varpi p S$ for Situation (i) and $\eps_{12},  \eps_{22} \in \varpi pS\subset I$ (using the fact that $v_p (bc^2), v_p (ac), v_p (bc)> 0$) and $\wt \eps_{32}\in cp S$ for Situation
\ref{situtaion-2}.

Now update $\cA$ by $\begin{pmatrix}
1  &  0   & 0 \\  0 & 1  & - b \sum\limits_{i \geq 1} \fb '_i \frac{E^{ip +r -s}}{p ^i} \\ 0 & 0 & 1 
\end{pmatrix} *_\varphi \cA $, we obtain $\begin{pmatrix}
\fd     &   1 + \eps_{12}   & c\fc'_0 + \eps_{13}    \\   b \fb'_0 E^r &  E^r \eps'_{22}  & E^r( 1+ \eps _{23} )  \\ E^s  & E^s\wt \eps_{32}  & E^s \eps_{33}  
\end{pmatrix}$
where $\eps_{j3} =  \alpha_j b\sum\limits_{i \geq 1}  \varphi (\fb' _i \frac{E^{ip +r -s}}{p ^i})\in \varpi pS\subset I $ where $\alpha_j$ from second column and $E^r \eps'_{22}  $ is updated by adding $  E^s \wt \eps_{32} \sum\limits_{i \geq 1}  b (\fb' _i \frac{E^{ip +r -s}}{p ^i}) \in \varpi p E^r S \subset  E^r I.$ 

The following two cases are under situation (i), that is, $c  \in \varpi \cO_F$.  

Case (1): If $v_p (b ')>1$ then $\bar {\cA}  = \begin{pmatrix}0 & 1 & 0  \\ 0  & 0 & u ^r  \\ u ^s  & 0  & 0 \end{pmatrix}$. 

Case (2): If $v_p (b')=1$ then let $\bar b' = \frac{b'}{p}\mod \varpi$ then   we have 	
$\bar {\cA}  = \begin{pmatrix}\bar b' u ^p  & 1 & 0 \\ 0  & 0 &  u ^ r \\ u ^s  & 0  & 0 \end{pmatrix}$.   

Now the real challenging case is when $v_p(b ') < 1$ for both situations, which we will assume from right now. Write $d : = \fd_{10} $, first note in this case $v_p (b') = v_p(d)$ by \eqref{eqn-b'} and then $\fd_1 = d \mu_{\fd}$ with $\mu_{\fd} \equiv 1 \mod \varpi \in \fS^\times$. If we write $\frac{1}{\fd_1}$ in the following then it should be understood as 
$\frac 1 d \mu_{\fd}^{-1}$. 
Our next row operation is to kill $E^s$-term.  Consider 
$\begin{pmatrix}
1  &  0   & 0 \\  0 & 1  & 0 \\ -\frac{p}{\fd_1} E^{s-p}  & 0 & 1 
\end{pmatrix} *_\varphi \cA $, we obtain 
\[
\cA = \begin{pmatrix}
\fd  + p\eps _{11}  &   1 + \eps_{12}  &   c \fc_0 ' + \eps _{13}\\  E^r (b\fb'_0 +  p  \eps_{21} )   &  E^r \eps'_{22}  & E^r ( 1+ \eps_{23} )  \\ -\frac{p}{\fd_1} E^{s-p} (a\fa_0'+  \sum\limits_{i \geq 2} \fd_i  \gamma_i (E) + p \eps_{31} )  & -\frac{p}{\fd_1} E^{s-p} (1 + \eps'_{32})     & - \frac{p}{\fd_1} E^{s-p} (c \fc'_0 + \eps'_{33})   
\end{pmatrix}
\]
where  $\eps'_{32}, \eps'_{33}$ are updated from row operation, and use the fact that $E^s = E^{s-p} E ^p \in p E ^{s-p}S$; Note that $\eps'_{32}$ is in $\varpi pS$ because $v_p (b' c)= v_p (dc)>0 $ for Situation \ref{situtaion-2}. 
All terms of $\eps_{i 1}$ comes from column operation which add $3$rd columns multiplying 
$\frac p d \varphi (E ^{s-p}\mu_\fd^{-1})$ to the first column. Note that $\varphi(E^{s-p})\in p ^ 2 S$, this allows we write $p \eps_{i1}$ in the first column. Also Situation \ref{situtaion-2} does not affect this claim because $   v_p (b' /c )= v_p (d/c) < 1 = v_p (p)$. 

 Now apply  $\begin{pmatrix}
1  &  0   & 0 \\  0 & 1  & 0 \\ 0   & 0 &  \frac{d}{p} 
\end{pmatrix} *_\varphi \cA $,  after updating   $\eps_{ij} $ from  $\eps'_{ij}$ if possible, 
we update  \begin{equation}\label{eqn-Matrix-II}
\cA = \begin{pmatrix}
\fd  + p\eps _{11}  &   1 + \eps_{12}  &   \frac p {d} (c \fc_0 ' + \eps _{13})\\  E^r (b \fb'_0 +  p  \eps_{21} )   &  E^r \eps_{22}  & \frac p {d} E^r ( 1+ \eps_{23} )  \\  -E^{s-p} \mu_\fd^{-1} (a\fa_0'+  \sum\limits_{i \geq 2} \fd_i  \gamma_i (E) + p \eps_{31} )  & - E^{s-p} \mu_\fd ^{-1}(1 + \eps_{32})     & - \frac{p}{\fd_1} E^{s-p} (c \fc' _0+ \eps_{33})   
\end{pmatrix}
\end{equation}

Now we first consider the case that $v_p (b) \leq  v_p (a)$. Since $v_p (b') = v_p (d) < 1 $, we have $v_p (b)< 1$ for both Situations (i) and (ii). Using $ \begin{pmatrix}
	b  &  0   & 0 \\  0 & 1  & 0 \\ 0   & 0 &  1 
\end{pmatrix} *_\varphi \cA $, we have 
$$\cA = \begin{pmatrix}
	\fd  + p\eps _{11}  &    b (1 + \eps_{12} )   &    b  \frac p {d} (c \fc_0 ' + \eps _{13}) \\  E^r (\fb'_0 +  \frac p b \eps_{21} )   &  E^r \eps_{22}  & \frac p {d} E^r ( 1+ \eps_{23} )  \\ - E^{s-p}\mu^{-1}_\fd (\frac a b \fa_0'+  \frac 1 b \sum\limits_{i \geq 2} \fd_i  \gamma_i (E) + \frac  p  b \eps_{31} )  & - E^{s-p}\mu_\fd ^{-1} (1 + \eps_{32})     & - \frac{p}{\fd_1}E ^{s-p} (c \fc'_0+ E \eps_{33})   
\end{pmatrix}. $$
Note that in $31$-entry  $\frac 1 b \sum\limits_{i \geq 2} \fd_i  \gamma_i (E) \in  \Fil^{2p} S$ in Situation (i) by \eqref{eqn-b'-II}. But in Situation \ref{situtaion-2}, we only have 
$\frac 1 b\sum\limits_{i \geq 2}\fd_i  \gamma_i (E) \in c \Fil^{2p} S$.   
We rewrite 21-entry by $E^r(\fb' _0 + \varepsilon_{21})$ with $\varepsilon_{21}\in I$. Apply $ \begin{pmatrix}
	1  & -\fd_1 \gamma_1 (E) E^{-r} (\fb_0 '  + \eps_{21})^{-1} &  0 \\  0 & 1  &0 \\ 0   & 0 &  1 
\end{pmatrix} *_\varphi \cA $, we obtain 
$$\cA = \begin{pmatrix}
	\sum\limits _{i \geq 2} \fd_i \gamma_i (E)  + p\eps _{11}  &  \eps'_{12} + b (1 + \eps_{12} )   &  -\mu _\fd \fb_0'  E^p  +   b  \frac p {d} (c \fc_0 ') + \eps' _{13}  \\  E^r (\fb'_0 + \eps_{21} )   &  E^r \eps'_{22}  & \frac p {d} E^r ( 1+ \eps_{23} )  \\ - E^{s-p} \mu _\fd ^{-1}(\frac a b \fa_0'+  \frac 1 b \sum\limits_{i \geq 2} \fd_i  \gamma_i (E) + \frac  p  b \eps_{31})  & - E^{s-p} \mu_\fd^{-1}(1 + \eps'_{32})     & - \frac{p}{\fd_1} E^{s-p} (c\fc '_0+ \eps_{33})   
\end{pmatrix}. $$
where $\eps'_{12}: =  -\fd_1 \gamma_1 (E)  (\fb_0' + \eps _{21})^{-1}\eps_{22} $ plus $\varphi (\fd_1 \gamma_1 (E) E^{-r} (\fb_0 ' + \eps _{21})^{-1} )$ of first column, and $\eps'_{22}, \eps'_{32}$ is updated similar way. Since $\varphi (\fd_1 \gamma_1 (E) E^{-r}) \in \varpi p S $, $\eps'_{12}, \eps'_{22} \in \varpi p S \subset I$, and $\eps'_{32} \in I$ (for Situation \ref{situtaion-2}, use $pc \in \varpi\cO_F$). Also $\eps'_{13}\in \varpi pS$ is also updated form row operation. Note that in $31$-entry  $E^{s-p}(\frac 1 b \sum\limits_{i \geq 2} \fd_i  \gamma_i (E) ) \in  J $ in both Situations so that we can use algorithm in \S \ref{subsec-algorithm}.  Using that $\fb_0' , \mu _\fd\equiv 1 \mod \varpi$ we conclude that for both Situations: 

Case (3): If $v_p (b) < v_p (a)$  then 
 $\bar \cA =   \begin{pmatrix}
	0  &  0  & - u ^p  \\  u ^r & 0  &  0  \\ 0    & - u ^{s-p } &  0 
\end{pmatrix}. $ 

Case (4): If $v_p (b) =  v_p (a)$  then 
$\bar \cA =   \begin{pmatrix}
	0  &  0  & - u ^p  \\  u ^r & 0  &  0  \\  \overline{a/b} u ^{s-p}    & - u ^{s-p } &  0 
\end{pmatrix}$ where $\overline{a/b} = a/b \mod \varpi$.

Now consider that $v_p (a)< v_p (b)$. We note that under this assumptions, the following fact holds for both Situations: 
\begin{equation}\label{Assump-1}
v_p (b c) > v_p (a). 
\end{equation}
Since $v_p (b') < 1$, we see that $v_p (a)= v_p (b')=v_p (d) < 1$ for both Situations. Starting $\cA $ in \eqref{eqn-Matrix-II}, 
 using $ \begin{pmatrix}
a  &  0   & 0 \\  0 & 1  & 0 \\ 0   & 0 &  1 
\end{pmatrix} *_\varphi \cA $, we have 
$$\cA = \begin{pmatrix}
	\fd  + p\eps _{11}  &    a (1 + \eps_{12} )   &    a ( \frac p {d} (c \fc_0 ' + \eps _{13})) \\ \frac 1 a E^r (b \fb'_0 +  p  \eps_{21} )   &  E^r \eps_{22}  & \frac p {d} E^r ( 1+ \eps_{23} )  \\ - E^{s-p} \mu^{-1}_\fd(\fa_0'+  \frac 1 a\sum\limits_{i \geq 2} \fd_i  \gamma_i (E) + \frac  p  a \eps_{31} )  & - E^{s-p}\mu^{-1}_\fd (1 + \eps_{32})     & - \frac{p}{\fd_1} E^{s-p} (c \fc'_0+ \eps_{33})  
\end{pmatrix}$$
By \eqref{eqn-b'-II}, \eqref{Assump-1}, we have $ \frac 1 a\sum\limits_{i \geq 2} \fd_i  \gamma_i (E) \in \Fil ^{2p} S$. Since $\frac p a \in \varpi \cO_F$ and $\fa'_0 \equiv 1 \mod p$ by Lemma \ref{lem-estimate},  we can rewrite $ \mu^{-1}_\fd (\fa_0'+  \frac 1 a\sum\limits_{i \geq 1} \fd'_i  \gamma_i (E) + \frac  p  a \eps_{31}) =1 + \delta_{31} $ with $\delta_{31}\in \Fil^{2p} S+ I+ p \fS$. Let us consider the transformation  
$ \begin{pmatrix}
1  &  0   &  \fd_1 \gamma_1 (E) E^{p-s} (1 + \delta _{31})^{-1} \\  0 & 1  & \frac 1 a E^{r+p-s} (b \fb'_0 + p \eps_{21})(1 + \delta _{31})^{-1} \\ 0   & 0 &  1 
\end{pmatrix} *_\varphi \cA $, we obtain: 
$$\cA = \begin{pmatrix}
a\fa'_0   + \sum\limits_{i \geq 2}\fd_i \gamma_ i (E) + p\eps _{11}  &   -d  \gamma_1 (E) +  a  + \varepsilon' _{12}   &  - E^p (c\fc'_0 ) + \frac{apc}{d} \fc'_0 + \varepsilon' _{13} \\ 0    &  E^r (- \frac {b}{a} \fb'_0 \mu _\fd ^{-1}  + \varepsilon'_{22})  &    \frac p {d} E^r ( 1 - \frac{bc}{a} \fb'_0\fc'_0 \mu _{\fd}^{-1}+  \varepsilon'_{23} )  \\ - E^{s-p} (1+  \delta_{31} )  & - E^{s-p} \mu_\fd ^{-1} (1 + \eps_{32})     & - \frac{p}{\fd_1} E^{s-p} (c \fc'_0+ \eps'_{33} ) 
\end{pmatrix}$$
where $ -d \gamma_1 (E) +  a  + \varepsilon'_{12} = a(1+ \eps_{12}) + \fd_1 \mu^{-1}_\fd\gamma_1 (E) (1 + \delta _{31})^{-1} (1+ \eps_{32})$ and 
$E^r (- \frac {b}{a} \fb'_0\mu ^{-1}_\fd  + \varepsilon'_{22})= -\frac 1 a E^{r} (b \fb'_0 + p \eps_{21})(1 + \delta _{31})^{-1} \mu _{\fd}^{-1}(1+ \eps_{32}) + E^r \eps_{22} $. 
 Using the fact that $(1 + \delta_{31})^{-1}-1 \in p \fS + \Fil^{2p}S + I$  and $\varpi p\fS  + \varpi  \Fil^{2p}S \subset I$, $\varepsilon'_{12}, \varepsilon'_{22} \in I$. 
The last column is updated from both row operation and columns operations: The column transformation just bring some extra elements in $I$ because  $\varphi (\gamma_1(E) E^{p-s})\in p ^3 S$ because $2p-s \geq 4$ and $\frac b   a \varphi (E^{r+p -s}) \in \varpi p ^2 S$. For row operations, the 13-entry is added by  $- E^p (1 + \delta _{31})^{-1}(c\fc'_0 + \varepsilon'_{33})$ and 23-entry is added by $ -\frac {p}{\fd_1 a} E^{r} (b \fb'_0 + p \eps_{21})(1 + \delta _{31})^{-1} (c\fc'_0 + \varepsilon'_{33})$, using that $E^p \in pS$ and $\frac {b c}{a}\in \varpi \cO_F$ by \eqref{Assump-1}, we see that the 13,  23-entries are added by $- E^p c\fc'_0 \mod I$, $-\frac {p}{\fd_1 a} E^{r} b \fb'_0 c\fc'_0 \mod \frac p d E^r I$ respectively.  


Now using $ \begin{pmatrix}
1  &  0   & 0 \\  0 & \frac d p  & 0 \\ 0   & 0 &  1 
\end{pmatrix} *_\varphi \cA $ and updating  22-entry by $-E^r (\frac b a \alpha + \eps_{22}) $ with $\eps_{22} \in I$, $\alpha = \fb_0 '\mu_\fd^{-1}\in \fS^\times$ and update other $\varepsilon'_{ij}$ by $\varepsilon_{ij}$ if necessary, 
we obtain
$$\cA = \begin{pmatrix}
	a\fa'_0   + \sum\limits_{i \geq 2}\fd_i \gamma_ i (E) + p\eps _{11}  &   -  E^p + \frac{p}{d} (a  + \eps_{12})   &  - E^p (c\fc'_0 ) + \frac{apc}{d} \fc'_0 + \eps _{13} \\ 0    &  E^r (- \frac {b}{a}\alpha   + \eps_{22})  &     E^r ( 1 - \frac{bc}{a} \fb'_0\fc'_0\mu^{-1}_\fd+  \eps_{23} )  \\ - E^{s-p} (1+  \delta_{31} )  & - \frac{p}{d}E^{s-p} \mu _\fd ^{-1}(1 + \eps_{32})     & - \frac{p}{\fd_1 } E^{s-p} (c \fc'_0+ \eps_{33} ) 
\end{pmatrix}$$
Now there is only $(13)$-entry which has term $-E^p (c\fc'_0 ) $ need to be eliminated for possible Situation \eqref{situtaion-2}. Let $\mu = 1 - \frac{bc}{a} \fb'_0\fc'_0 \mu^{-1}_\fd+  \eps_{23} \in S^\times$. Note that $\mu^{-1} -1 \in I + \varpi \fS$.  
So we use 
 $ \begin{pmatrix}
	1  &  E^{p-r} \mu^{-1} (c  \fc'_0) 
	   & 0 \\  0 & 1  & 0 \\ 0   & 0 &  1 
\end{pmatrix} *_\varphi \cA $. Then we update 
$$\cA = \begin{pmatrix}
	a\fa'_0   + \sum\limits_{i \geq 2}\fd_i \gamma_ i (E) + p\eps _{11}  &   -  E^p (1 + \mu^{-1} c \fc'_0 (\frac b a\alpha - \eps_{22})) +  \frac{p}{d} (a  + \eps'_{12} )   &   \frac{apc}{d} \fc'_0 + \eps _{13} \\ 0    &  E^r (- \frac {b}{a} \alpha  + \eps_{22})  &     E^r \mu   \\ - E^{s-p} (1+  \delta_{31} )  & - \frac{p}{d}E^{s-p} \mu^{-1}_\fd (1 + \eps'_{32})     & - \frac{p}{\fd_1} E^{s-p} (c \fc'_0+ \eps_{33} ) 
\end{pmatrix}$$
 Note that $\varphi (E ^{p-r}c)\in p ^2 c S \subset \frac p d (\varpi pS)$ also for Situation \ref{situtaion-2} (using $v_p (a) = v_p (d)$ and $v_p (ac)>0 $). This allows us update $ \eps_{12}$, $\eps_{32}$ to $\eps'_{12}$, $\eps'_{32}$ respectively by column transformation. Now $\cA$  satisfies the requirement of Proposition \ref{prop-algorithm-OK}. Then 
 $\bar \cA =   \begin{pmatrix}
	0  &  - u ^p   & 0 \\  0 & 0   &  u ^r \\- u ^{s-p }   & 0 &  0 
\end{pmatrix}$. 

\section{Residue Galois representations} \label{sec-Galoisrep}
In this section, we discuss all possible semi-simplification of $\Vbar$. Since we will discuss all types $A$ in Theorem \ref{them-classification-SD-lattices}, we will use 
${\rm I} (a_1, b_1 , c_1)$, ${\rm II} (a_2, b_2, c_2)$, ${\rm III} (a_3, b_3, c_3)$ to distinguish parameters from   different types. Similar notations applies to parameters $b'$ in the previous sections

\subsection{Functor $V^*_{\F}$} We collect some information on $V_\F^*$ to compute the reduction $\Vbar$ of $V$.  Pick $p _i \in \Qpbar$ so that $p_0 = -p$ and $p _{i +1}^p = p_ i$. Let $\Q_{p, \infty} = \bigcup_{n =1}^\infty \Q_p(p _i)$ and $G_\infty : = \Gal (\Qpbar /\Q_{p, \infty})$. Let $\underline{p}= (p _i)_{i \geq 0} \in \cO_{\C_p}^\flat$. Write $k= \F_p = \Z/ p \Z$.  The embedding 
$\ku \to \cO _{\C_p}^\flat$ via $u \mapsto \underline p$ is compatible with $\varphi$ and $G_\infty$. Let $(\fN , \varphi_\fN)$ be a Kisin module killed by $p$, which means that $\fN $ is a finite free $\F [\![u]\!]$-module together a semi-linear $\varphi_{\ku}$-linear, $\F$-linear  endomorphism $\varphi_{\fN  }: \fN  \to \fN $ so that the cokernel of linearization $1\otimes \varphi _{\fN }: \ku \otimes_{\varphi, \ku} \fN \to \fN  $ is killed by $u ^h$. Then we define 
$$V^*_\F(\fN  ): = \Hom_{\ku , \varphi} ( \fN , \cO_{\C_p}^\flat). $$
It is well-known that $V^*_{\F}(\fN)$ is an $\F$-representation of $G_\infty$,  $V^*_\F (\fN)$ only depends on the attached \'etale $\varphi$-module $\fN [\frac 1 u]$ and $V^*_\F(\fN  )= \Hom_{\ku[\frac 1 u] , \varphi} ( \fN[\frac 1 u] , \C_p^\flat) $. In particular, let $T \subset V$ be a $G_\infty$-stable $\Z_p$-lattice with $\M= \M(T) $ the attached Kisin module and set $\overline \M: = \M / \varpi \M$,   then $V_\F^* (\overline \M)\simeq T/ \varpi T$ by \cite[Cor. 2.3.2]{BergdallLevin-BLZ}. As discussed \emph{loc.cit}, the semi-simplification of $\overline V$ is the same as that of $V_\F^* (\overline \M)$. 
Write $I _{\Q_p}$ the inertia subgroup of $G_{\Q_p}$
and $\omega_d: I_{\Q_p}\to \overline{\F_p}^\times$ the fundamental character of level $d$. Recall that $\omega_d (g)= \frac{g (\pi )}{\pi}\mod p$ where $\pi$ is a fixed $p ^d-1$-th root of $-p$. 
\begin{prop}\label{prop-inertia-weight} Let $\fN $ be a Kisin module killed by $p$ with basis $e_1 , \dots , e_d$. Assume that $\varphi (e_i) = u ^{a_i} e_{i +1}$ with $e_{d+1} = e_1$ and $ 0 \leq a_i \leq p-1 $. If $V_{\F}^* (\fN)$ is irreducible then $V_\F^* (\fN)$ is an induction of a  character with weight $(a_{d}, a_{d-1}, \dots, a_1)$ of level $d$. 
\end{prop}
\begin{proof} This is well-known, for example, in \cite[Prop.3.1.2]{LLL}. For convenience of readers, we sketch the proof here. 
	Let $X =(x_1, \dots, x_d) \in V^*_\F (\fN): = \Hom_{\ku, \varphi} (\fN , \cO ^\flat_{\C_p} ) $ so that $\varphi(x_i) = \underline p ^{a_i} x_{i +1}$. Then we see that 
	$\varphi ^d (x_1) = \underline p  ^{\sum_{j = 1}^d p  ^{d-j} a_j} x_1$. So $x_1$ is a solution of 
	\begin{equation} \label{eqn-tame-inertia}
	x^{p ^d}= \underline p^{\sum_{j = 1}^d p ^{d-j} a_j} x . 
	\end{equation}
	Pick $\upi : = (\pi_i)_{i \geq 0}\in \cO_{\C_p}^\flat$ so that $\pi_0= \pi$ and $\upi^{p ^d-1} = \underline{p}$. Then we can select $x_1 = \upi ^{\sum_{j = 1}^d p ^{d-j} a_j}$ and all other roots in \eqref{eqn-tame-inertia} are $\zeta_{p ^d -1} x_1$ with $\zeta_{p ^d -1}\in \overline{\F_p}^\times$ a $p ^d -1$-th  roots of unity. Since $X$ is determined by $x_1$,  $g(X)$ is determined by $g(x_1)$ for any $g\in G_\infty$. In particular, the $G_\infty$-action on $V^*_\F(\fN)$ factors through $\Gamma_K: = \Gal (K /\Q_p)$, where $K$ is the splitting field of $X^{p^d -1} + p$. Note that $\Gamma _K = \langle \tau \rangle \rtimes I_K$ where $I_K\subset \Gamma_K$ is the inertia subgroup and $\tau\in \Gamma_K$ is a lift of Frobenius. Now proposition follows that for any $g \in I_K$, since $g (\upi) = \omega_d(g) \upi$,  we have that $g (X) = \omega ^{\sum_{j = 1}^d p  ^{d-j} a_j}_d (g)X$. 
\end{proof}

\begin{lemma}\label{lem-row-simplify}Let $A \in \rM _d ^h (\ku)$ with $h \leq 2p -3$ and $X \equiv I _d \mod u ^2$. Then there exists a $ Y \in \GL_d (\ku)$ so that $Y*_\varphi A = XA$. 
\end{lemma}
\begin{proof} Consider $X *_\varphi A = X A \varphi (X^{-1})$. Since $X \equiv I _d \mod u ^2$. We see that $X *_\varphi A = X A + u ^{2p}Z$. Since $A \in \rM ^h_ d (\ku)$, there exists $B$ so that $AB = BA = u ^h I _d$. Hence $$X*_\varphi A = XA +  u ^{2p - h}Z B X^{-1} X A = ( I _d + u ^{h_1} Z_1) XA $$ with $h_1 = 2p -h$ and $Z_1 = ZB X^{-1}$. Let $Y_1 = (I _d + u ^{h_1} Z_1) ^{-1}$ and consider $(Y_1 X) *_\varphi A $, similarly we obtain that $(Y_1 X) *_\varphi A = (I _d + u^{h_2} Z_2) XA$ with $h_2 = p h_1 - h$. Continue this step, we see that $(Y_n Y_{n -1}\cdots Y_1 X)*_\varphi A = (I_d + u ^{h_n} Z_n) XA$ with $h_n = p h_{n -1}-h$. It is clear that $\lim\limits_{n \to \infty} h _n = \infty$ and then 
	$Y = \prod_{n = 1} ^\infty Y_n X$ is required.  
\end{proof}

Let $\Q_{p^f}$ denote the unramified extension of $\Q_p$ with degree $f$. 
Now by the calculation of the previous section, for each types and cases, we have obtained Kisin module $\overline \M$ killed by $p$ by finding the matrix $\bar \cA$ of $\varphi$. Now we can use Proposition \ref{prop-inertia-weight} and Lemma \ref{lem-row-simplify} to calculate semi-simplification of $V^* _\F (\overline \M) $, which is also the semi-simplification $\overline V$ of reduction of $V$. Here we only show the computation in details for cases in \S \ref{subsubsec-first-case} and calculation for  other types and cases are the similar. The first matrix  $\bar \cA$ in \S \ref{subsubsec-first-case} is $\bar \cA = \begin{pmatrix} 0 & 0 & 1 \\ u ^r & 0 & 0 \\ 0 & u ^s & 0 \end{pmatrix}$. Let $Y = [ 1 , u , 1]$. Then $Y *_\varphi \bar \cA = \bar \cA = \begin{pmatrix} 0 & 0 & 1 \\ u ^{r+1} & 0 & 0 \\ 0 & u ^{s-p} & 0 \end{pmatrix}$. Now Proposition \ref{prop-inertia-weight} shows that $\Vbar = \Ind^{G_{\Q_p}}_{G_{\Q_{p ^3}}} \omega_3^{ (s-p) p + (r+1) p ^2}$. For the second matrix $\bar \cA = \begin{pmatrix} 0 & 0 & 1 \\ u ^r & \bar c u^r  & \bar \alpha u ^{r+p} \\ 0 & u ^s & 0 \end{pmatrix}$,  we swap the first and second rows:  $\begin{pmatrix} 0 & 1 & 0 \\ 1 & 0 & 0 \\ 0 & 0 & 1 \end{pmatrix}*_\varphi\bar \cA =  \begin{pmatrix} \bar c u^r & u ^r  & \bar \alpha u ^{r+p}  \\ 0  & 0 & 1 \\ u ^s  & 0 & 0 \end{pmatrix} $. Now update $\bar \cA$ and consider  $ X= \begin{pmatrix} 1 & 0   & 0   \\ 0  & 1 & 0 \\  - {\bar c}^{-1} u ^{s-r}  & 0 & 1 \end{pmatrix}\begin{pmatrix} 1 & 0   & -\bar \alpha u ^{r+p}  \\ 0  & 1 & 0 \\ 0  & 0 & 1 \end{pmatrix} $. By Lemma \ref{lem-row-simplify}, we can update $\bar \cA$ by $X\bar \cA =\begin{pmatrix} \bar c u^r & u ^r  & 0 \\ 0  & 0 & 1 \\ 0   & {\bar c}^{-1} u ^s & 0 \end{pmatrix} $. Now $V^*_\F(\overline \M)$ has two irreducible parts: The first one is determined by $\varphi (e_1) = \bar c u ^r$. So it is $\chi_{\bar c}\omega_1^r$ where $\chi _{\bar c}$ is the unramified character which sends Frobenius to $\bar c$; The second one is determined by matrix $\begin{pmatrix} 0 & 1 \\ \bar c^{-1} u ^s & 0
\end{pmatrix}$. Let $x\in \overline \F_p$ so that $x^2 = \bar c$. We have $\begin{pmatrix} u  & 0 \\0 & x\end{pmatrix}*_\varphi \begin{pmatrix} 0 & 1 \\ \bar c^{-1} u ^{s} & 0
\end{pmatrix}= x^{-1} \begin{pmatrix} 0 & u \\ u ^{s-p} & 0 \end{pmatrix}$. So this part is $ \chi_{x^{-1}} \Ind^{G_{\Q_p}}_{G_{\Q_{p ^2}}} \omega_2^{1 + (s-p) p }= \chi_{x^{-1}} \Ind^{G_{\Q_p}}_{G_{\Q_{p ^2}}} \omega_2^{sp}$. In summary, for the second matrix in \S \ref{subsubsec-first-case}, $\Vbar= \chi \Ind^{G_{\Q_p}}_{G_{\Q_{p ^2}}} \omega_2^{p + (s-p) p } \oplus \chi^{-2} \omega_1 ^{r}$ with $\chi = \chi_{x^{-1}}$. 

Here we also explain why unit $\bar \nu_{32}\equiv 1 \mod u ^p\in \F[\![u]\!]^\times$ in \S \ref{subsubsec-hardest} can be ignored from $\bar \cA$ 0. Indeed, set $z = \prod\limits_{i = 0}^\infty \varphi^{3i} (\bar \nu_{32}^{-1})$, 
$x= \varphi (z), y = \varphi ^2 (z)$.  Then we have 
$$[x, y , z]*_{\varphi} \begin{pmatrix} 0 &  0 & - u ^p \\ u ^r & 0 & 0 \\ 0 & - u ^{s-p} \bar \nu_{32} & 0 \end{pmatrix}= \begin{pmatrix} 0 &  0 & - u ^p \\ u ^r & 0 & 0 \\ 0 & - u ^{s-p}  & 0 \end{pmatrix}. $$

Now we summarize all types of $\overline V$ by two categories: reducible cases and irreducible cases in the next two subsections. Since $b'$ plays very important role in each types, we recall their definitions here. 
$$b'_1 = (s-r)b_1 - sa_1c_1; \ \  b'_2 = ra_2 + (s-r) b_2 c_2; \ \ b'_3 = (s-r)b_3 c_3 -s a_3 \mu.  $$

\subsection{Reducible reduction}\label{subsec-reducible-red} In the following we use $\chi$ to denote a unramified character. Note that in this following cases, $\chi$ may not be the same unless it appears in one equation.  
\subsubsection{$\Vbar=  \chi \Ind^{G_{\Q_p}}_{G_{\Q_{p ^2}}} \omega_2^{p + r p } \oplus \chi^{-2} \omega_1 ^{s-p}$} This happens for Type I and $v_p (b _1' )= 1$.
\subsubsection{$\Vbar= \chi \Ind^{G_{\Q_p}}_{G_{\Q_{p ^2}}} \omega_2^{p + (s-p) p } \oplus \chi^{-2} \omega_1 ^{r}$} This happen for 
\begin{enumerate}
\item 	Type I when $V_p(b'_1)< 1$ and $v_p (a_1) = v_p (b_1)$ . 
\item Type II when $V_p(b'_2)< 1$ and $v_p (a_2) = v_p (b_2)$ .
\item Type III when $V_p (b'_3)< 1$ and $v_p (a_3 ) \leq b _3 c_3$ but $v_p (\mu -\frac{b_3c_3}{a_3})= 0$. 	
\end{enumerate}

\subsubsection{$\Vbar= \chi \Ind^{G_{\Q_p}}_{G_{\Q_{p ^2}}} \omega_2^{r  + (s-p) p } \oplus \chi^{-2} \omega_1 ^{p}$} This happen for Type II  when $V_p(b'_2)=  1$.  
\subsubsection{$\Vbar= \chi \Ind^{G_{\Q_p}}_{G_{\Q_{p ^2}}} \omega_2^{ s p } \oplus \chi^{-2} \omega_1 ^{r}$} This happen for Type III   when $V_p(b'_3)>  1$ and Type I when $v_p (b '_1)=1$.
\subsubsection{$\Vbar= \chi_1 \chi _2 \omega_1 ^{p }\oplus  \chi_2^{-1} \omega_1^{s-p} \oplus \chi_1 ^{-1} \omega_1 ^{r}$} This happen for Type III   when $V_p(b'_3)= 1$.

\subsection{Irreducible reductions}\label{subsec-irreducible-red} We list all cases when $\overline V$ is irreducible. 
\subsubsection{$\Vbar = \Ind^{G_{\Q_p}}_{G_{\Q_{p ^3}}} \omega_3^{ (s-p) p + (r+1) p ^2}$}\label{subsubsec-irr-1} This happens for Type I when $v_p (b'_1) > 1$. 
\subsubsection{$\Vbar = \Ind ^{G_{\Q_p}}_{G_{\Q_{p ^3}}}\omega_3^{ r p + s p ^2}$}\label{subsubsec-irr-2} This happens for Type II when $v_p (b'_2) > 1$.
\subsubsection{$\Vbar = \Ind ^{G_{\Q_p}}_{G_{\Q_{p ^3}}}\omega_3 ^{p + rp + (s-p) p ^2} $}\label{subsubsec-hard-1} This happens for 
\begin{enumerate}
	\item Type I, $v_p(b'_1) < 1$, $v_p (a_1) < v_p (b_1)$
	\item Type II, $v_p (b'_2) <1$, $v_p (a_2) < v_p (b_2)$. 
	\item Type III, $v_p (a_3) = v_p (b_3 c_3) < 1$ and $ v_p (b_3 c_3 / a _3 -  \mu )> 0$. Note in this case, we have $b'_3 = (s-r) b_3 c_3 - s a_3 \mu$ satisfies 
	$v _p (b'_3)=v_p (r b_3 c_3) < 1$ because $v_p (b_3 c_3 - a_3 \mu) > v_p (b_3 c_3 )$. 
\end{enumerate}
\subsubsection{$\Vbar = \Ind ^{G_{\Q_p}}_{G_{\Q_{p ^3}}}\omega_3 ^{p + (s-p) p + r p ^2}$}\label{subsubsec-hard-2} This happens when 
\begin{enumerate}
	\item Type I, $v_p(b'_1) < 1$, $v_p (a_1) >  v_p (b_1)$
	\item Type II, $v_p (b'_2) <1$, $v_p (a_2) >  v_p (b_2)$. 
	\item Type III, $v_p (b_3 c_3) < 1$, $v_p (a_3)> v_p (b_3 c_3 ). $ Note in this case, $v_p (b'_3) = v_p (b_3 c_3) < 1$. 
\end{enumerate}

Finally, we can use the above computation to study universal crystalline deformation ring constructed in \cite{kisin4}. Let $\bar \rho : G_{\Q_p} \to \GL_3 (\F)$ be a residue representation. If $\bar \rho$ is irreducible then by \cite{kisin4} for Hodge-Tate type ${\bf v}: =\{0 , r, s\}$ there exists the universal crystalline deformation ring $R_{\bar \rho} ^{\bf v}$.

\begin{thm}\label{thm-deformation} Let  $\rho: G_{\Q_p} \to \GL_3 (\cO_F)$ be a crystalline representation with Hodge-Tate weights $\{0, r, s\}$ satisfying $2 \leq r \leq p-2, \ \  2+p \leq s \leq r + p-2. $	Assume that the reduction $\bar \rho: = \rho \mod \varpi $ is irreducible. Then $\Spec( R^{\bf v}_{\bar \rho}[\frac 1 p])$ is connected. 
\end{thm}
\begin{proof} Let $R, S $ be $\Z_p$-algebras, set $R (S):  = \Hom_{\Z_p-{\rm algebra}} (R ,S )$. 
Given $x_1$, $x_2 \in R^{\bf v}_{\bar \rho}(\cO_{F})$ so that $\rho_{x_i}: G _{\Q_p}\to  \GL_3 (\cO_F)$  are two $\cO_F$-lattices in crystalline representations,  we will show that $x_i$ are connected by several $\Spec (\Lambda_i)$ so that $\Lambda_i$ are quotients of $R^{\bf v}_{\bar \rho}$ and $\Lambda_i$ are $\cO_F$-domains. 

To be more precise,  write $\Lambda= \Lambda_i$. We will construct  Kisin module $\M_{\Lambda}$  satisfying the following: 

\begin{enumerate}
\item For every  $y \in \Lambda(\cO_{\overline {\Q}_p})$, $\M_y: = \M_\Lambda \otimes_{\Lambda, y }\cO_{\overline \Q_p}$ is the unique Kisin module for a crystalline representation  assumed in Theorem. 

\item  $\Lambda$ is a complete local neotherian $\cO_F$-domain. 
\end{enumerate}  
It is well-known that $\M_\Lambda$ defines a representation $\rho _\Lambda: G_{\infty} \to \GL_3 (\Lambda)$. Since $\bar \rho|_{G_{\infty}} $ is irreducible if and only if $\bar \rho$ is irreducible, by \cite{Kim-thesis}, the universal deformation ring of $ R_{\infty}$ which parameterize $G_{\infty}$-deformation of $\bar \rho|_{G_{\Q_p}}$ with finite $E$-height $s$ does exist,  By \cite[Prop 3.9, Lem. 3.10]{LLHLM}, if ${\rm ad} \bar \rho$ is cyclotomic free then  $R^{\bf v }_{\bar \rho}$  is a  quotient of  $R_{\infty}$. In our situations, $\bar \rho|_{I_{\Q_{p}}}\simeq  \bigoplus\limits_{i = 0}^2\omega_3^{a^{p ^i}}$ with $a = (s-p) p + (r+1) p ^2,  r p + sp ^2, (r+1) p  + (s-p) p ^2, (s-p +1) + r p ^2$.  Then ${\rm ad} \bar \rho\simeq  \bigoplus\limits_{i , j = 0} ^2 \omega_3^{a^{p ^i- p ^j}}$. It suffices to show that $a ^{p ^i - p ^j}\not \equiv  p ^2 \mod p ^3 -1$ and this can be verified by the assumption that $2 \leq r \leq p-2, \ \  2+p \leq s \leq r + p-2.$

Since $\Lambda(\cO _{\overline \Q_p}) \subset R^{\bf v}_{\bar \rho } (\cO_{\overline \Q_p})$ by (1) in the above, $ R^{\bf v}_{\bar \rho}$ and $\Lambda$ are both $\Z_p$-flat and reduced, we conclude that the natural map $\pi: R_\infty \to \Lambda$  factors through $R ^{\bf v}_{\bar \rho}.$ In particular, since $\Lambda$ is a domain,  $\pi$ induces a connected subscheme of $\Spec (R ^{\bf v}_{\bar \rho}[\frac 1 p])$. Now our aim is to construct several such $\Lambda_i$ so that  $x_j\in \Lambda_{i_j} (\cO _{\Qpbar})$ for $j =1 ,2$. 

Now let us discuss each cases in \S \ref{subsec-irreducible-red}. 
First consider the case in \S \ref{subsubsec-irr-1}: $\Vbar = \Ind^{G_{\Q_p}}_{G_{\Q_{p ^3}}} \omega_3^{ (s-p) p + (r+1) p ^2}$. In this case, both $D_{\cris}(\rho_{x_i})$ have strongly divisible lattices in Type I (we simply say $x_i$ belongs to Type I in the following) and satisfying $b ' > 1$. Set $\Lambda= \cO_F[\![ a_1 , b_1 , c_1 , z ]\!]/ (b_1' - p z)$. Clearly, by calculation in \S \ref{subsubsec-first-case} and Proposition \ref{prop-algorithm-lambda}, we obtain required Kisin module $\M _\Lambda$ satisfying the above conditions and both $x_i$ are in $\Lambda (\cO _F)$. Hence $\Spec (R ^{\bf v}_{\bar \rho}[\frac 1 p])$ is connected in this case. The proof for case in  \S \ref{subsubsec-irr-2} is the same. 

The remaining two cases are much harder.  Let us treat the case in \S \ref{subsubsec-hard-1}. We first show that if $x_1$, $x_2$ are both in Type I or Type II, then they are in the same connected component. First suppose that $x_i$ belongs to  Type I so that $v_p (b'_1)<1$ and $v_p (a_1) < v_p (b_1)$, then we may consider $\Lambda= \cO _F [\![a_1,   c_1, w, z]\!]/ (pz - a_1 ((s-r)w -s c_1))$ where $w= \frac{b_1}{a_1}$ and hence $b_1' = (s-r)b_1 - sa_1c_1 = a_1 ((s-r)w -s c_1)$. According to the computations in \S \ref{subsubsec-I-1}, we construct a Kisin module $\fM _\Lambda$ so that $x_i \in \Lambda (\cO_{\Qpbar})$. Hence $x_i$ are inside one connected component of $\Spec (R^{\bf v}_{\bar \rho}[\frac 1 p])$.  If $x_i$ belongs to Type II so that $v_p (b_2') < 1$ and $v_p (a_2)< v_p (b_2)$, then we use the similar argument via  computation after \eqref{Assump-1}, and conclude that $x_i$ are connected in this case. 

To treat the case that  $x_i$ from different Types or one of $x_i$ in Type III,  we use \S \ref{subsec-relations} and \S \ref{subsec-TypeII} Situation \eqref{situtaion-2}. We will show that there \emph{exists} a $x_1$ in Type I connects to a $x_2$ in Type II, and any $x_1$ in Type III connects to  a $x_2$ in Type II. 
Now first assume that $x_1$ is in Type I, since all $x_i$ in Type I are connected, we may select $x_1 = {\rm I} (a_1, b _1 , c_1)$ so that $v_p (a_1c_1 )< 1$, $v_p (a_1)> v_p (c_1)$ and $b_1 = a_1c_1$. According to \S \ref{subsubsec-I to II}, we see that ${\rm I}(a_1, b_1, c_1) = {\rm II}(a_2 , b_2 , c_2)$ with $a_2 = a_1 + c_1 p ^r$, $b_2 = b_1 -a_1 c_1 + c_1^{-1}p ^{s-r}= c_{1}^{-1}p ^{s-r}$ and $c_2 = c_1^{-1}$.  We easily check that $a_2, b_2 , c_2$ are in the Situation \eqref{situtaion-2}(b). Then the calculation for  Situation \eqref{situtaion-2}(b) shows that $x_1$ is in the same connect components of $x' = {\rm II} (a_2, b _2, c'_2)$ with $c'_2 \in \varpi \cO _F$. To be more precise, pick  $\alpha\in F$ so that $v_p (\alpha) = \min  \{ v_p (b_2 / a_2), v_p (p /a_2), v_p ( a_2)\}$. Note that $v_p (\alpha )>0 $ and $v_p (c_2 \alpha)> 0$ by our calculation the above.  Now set $\Lambda= \cO_F[\![z]\!]$, the calculation for Situation \eqref{situtaion-2} (b) show that there exists a Kisin module $\fM_{\Lambda}$ by let $c _2 = \frac{z}{\alpha}$ so that 
$x_1 \in \Lambda(\cO_{\Qpbar})$ and $x_2 = {\rm II}(a_2 , b_2 , c'_2)\in \Lambda(\cO_{\Qpbar})$ with $c'_2 \in \varpi \cO_F$. This complete a proof that any $x_1$ in Type I connects a $x_2$ in Type II, hence all $x_2$ in Type II.

If $x_1$ is in type III, then by \S \ref{subsubsec-I to II},   we have ${\rm III}(a_3, b_3, c_3) = {\rm II}(a_3 + \mu p ^r , b_3 c_3 -a_3 \mu , \mu ^{-1})$. The condition in \S \ref{subsubsec-hard-1} (3) exactly shows that $x_1$ is in Situation \eqref{situtaion-2} (b). Then the same argument as the above shows that $x_1$ connects to a $x_2 = {\rm II} ((a_3 + \mu p ^r , b_3 c_3 -a_3 \mu , c_2)$ with $c_2 \in \varpi \cO_F$. Hence any $x_1$ in Type III connects to a $x_2$ in Type II and hence all $x_2$ in Type II. 

Finally we treat the case in \S \ref{subsubsec-hard-2}. Our strategy is the  same as before. If both  $x_i$ belongs to  Type I so that $v_p (b'_1)<1$ and $v_p (a_1) >  v_p (b_1)$, then consider $\Lambda= \cO _F [\![b_1,   c_1, w, z]\!]/ (pz - b_1 ((s-r) -s w c_1))$ where $w= \frac{a_1}{b_1}$ and hence $b_1' = (s-r)b_1 - sa_1c_1 = b_1 ((s-r) -s  w c_1)$. According to the computations in \S \ref{subsubsec-I-1}, we construct a Kisin module $\fM _\Lambda$ so that $x_i \in \Lambda (\cO_{\Qpbar})$. Hence $x_i$ are inside one connected component of $\Spec (R^{\bf v}_{\bar \rho}[\frac 1 p])$.  If $x_i$ belongs to Type II so that $v_p (b_2') < 1$ and $v_p (b_2)< v_p (a_2)$, then we use the similar argument via  computation before \eqref{Assump-1}, and conclude that $x_i$ are connected in this case. 

To treat the case that  $x_i$ from different Types or one of $x_i$ in Type III,  we use the \ref{subsec-relations} and \ref{subsec-TypeII} Situation \eqref{situtaion-2} (a). 
Now first assume that $x_1$ is in Type I, since all $x_i$ in Type I are connected, we may select $x_1 = {\rm I} (a_1, b _1 , c_1)$ so that $v_p (b_1)< v_p (a_1)< 1$, $v_p (b _1)> v_p (c_1^2 )$. By  \S \ref{subsubsec-I to II}, we see that ${\rm I}(a_1, b_1, c_1) = {\rm II}(a_2 , b_2 , c_2)$ with $a_2 = a_1 + c_1 p ^r$, $b_2 = b_1 -a_1 c_1 + c_1^{-1}p ^{s-r}$ and $c_2 = c_1^{-1}$.  We easily check that $a_2, b_2 , c_2$ are in the Situation \eqref{situtaion-2}(a). Pick  $\alpha\in F$ so that $v_p (\alpha) \leq \min  \{ \frac{v_p (b_2 )}{2} \}$. Note that $v_p (\alpha )>0 $ and $v_p (c_2 \alpha)> 0$ by our calculation the above.  Now set $\Lambda= \cO_F[\![z]\!]$, the calculation for Situation \eqref{situtaion-2} (a) show that there exists a Kisin module $\fM_{\Lambda}$ by letting $c _2 = \frac{z}{\alpha}$ so that 
$x_1 \in \Lambda(\cO_{\Qpbar})$ and $x_2 = {\rm II}(a_2 , b_2 , c'_2)\in \Lambda(\cO_{\Qpbar})$ with $c'_2 \in \varpi \cO_F$. This complete the  proof that any $x_1$ in Type I connects a $x_2$ in Type II, hence all $x_2$ in Type II. 
If $x_1$ is in type III, then by \S \ref{subsubsec-I to II},   we have ${\rm III}(a_3, b_3, c_3) = {\rm II}(a_3 + \mu p ^r , b_3 c_3 -a_3 \mu , \mu ^{-1})$. The condition in \S \ref{subsubsec-hard-2} (3) exactly shows that $x_1$ is in Situation \eqref{situtaion-2} (a). Then the same argument as the above shows that $x_1$ connects to a $x_2 = {\rm II} (a_3 + \mu p ^r , b_3 c_3 -a_3 \mu , c_2)$ with $c_2 \in \varpi \cO_F$. Hence any $x_1$ in Type III connects to a $x_2$ in Type II and hence all $x_2$ in Type II. 
\end{proof}


\bibliography{semistable_bibliography}
\bibliographystyle{abbrv}

\end{document}